\documentclass[12pt]{amsart}
\usepackage{amsmath,amsfonts,amssymb,graphicx,hyperref}
\usepackage{amsthm}

\newcommand{\Z}{\mathbb{Z}}
\newcommand{\R}{\mathbb{R}}
\newcommand{\N}{\mathbb{N}}
\newcommand{\Q}{\mathbb{Q}}
\newcommand{\C}{\mathbb{C}}

\newcommand{\paren}[1]{\ensuremath{\left( #1 \right)}}
\newcommand{\set}[1]{\ensuremath{\left\{ #1 \right\}}}
\newcommand{\innerprod}[1]{\ensuremath{\left< #1 \right>}}
\newcommand{\norm}[1]{\ensuremath{\left\| #1 \right\|}}
\newcommand{\abs}[1]{\ensuremath{\left| #1 \right|}}
\newcommand{\setdiv}{\,\middle|\,}
\newcommand{\e}[1]{e\paren{#1}}
\newcommand{\summod}[1]{\ensuremath{\,(\mathrm{mod}\,#1)}}

\newcommand{\pderv}[3][]{{\frac{\partial^{#1}#2}{\partial{#3}^{#1}}}}

\newcommand{\Matrix}[1]{\begin{pmatrix}#1\end{pmatrix}}
\newcommand{\SmallMatrix}[1]{\left(\begin{smallmatrix}#1\end{smallmatrix}\right)}
\newcommand\revdots{\mathinner{\mkern1mu\raise1pt\vbox{\kern7pt\hbox{.}}\mkern2mu\raise4pt\hbox{.}\mkern2mu \raise7pt\hbox{.}\mkern1mu}}

\newcommand{\Max}[1]{\ensuremath{\max \set{#1}}}
\newcommand{\Min}[1]{\ensuremath{\min \set{#1}}}
\newcommand{\floor}[1]{\ensuremath{\left\lfloor #1 \right\rfloor}}

\newcommand{\piecewise}[1]{\left\{\begin{matrix}#1\end{matrix}\right.}
\newcommand{\If}{\mbox{if }}
\newcommand{\Otherwise}{\mbox{otherwise}}

\renewcommand{\Re}{{\mathop{\mathgroup\symoperators Re}}}
\renewcommand{\Im}{{\mathop{\mathgroup\symoperators Im}}}
\newcommand{\sgn}{{\mathop{\mathgroup\symoperators \,sgn}}}

\newcommand{\wbar}[1]{\overline{#1}}
\newcommand{\wtilde}[1]{\widetilde{#1}}
\newcommand{\what}[1]{\widehat{#1}}

\newcommand{\trans}[1]{{#1}^T}

\theoremstyle{plain} 
\newtheorem{thm}{Theorem}
\newtheorem{cor}[thm]{Corollary}

\newtheorem{prop}[thm]{Proposition}

\theoremstyle{remark}

\DeclareMathAlphabet{\mathcalligra}{T1}{calligra}{m}{n}
\newcommand{\WigDName}{\mathcal{D}}
\newcommand{\WigDMat}[1]{\WigDName^{#1}}
\newcommand{\WigDRow}[2]{\WigDMat{#1}_{#2}}
\newcommand{\WigD}[3]{\WigDMat{#1}_{#2,#3}}
\newcommand{\WigdName}{\mathcalligra{d}}
\newcommand{\Wigd}[3]{\WigdName^{#1}_{#2,#3}}

\DeclareMathOperator{\Tr}{Tr}
\DeclareMathOperator*{\res}{res}
\DeclareMathOperator{\diag}{diag}

\newcommand{\tildek}[1]{\wtilde{k}\paren{#1}}
\newcommand{\Dtildek}[2]{\mathcal{R}^{#1}\paren{#2}}
\newcommand{\dualWigDMat}[1]{\wtilde{\WigDName}^{#1}}
\newcommand{\dualWigDRow}[2]{\dualWigDMat{#1}_{#2}}

\newcommand{\AdSq}{{\mathrm{Ad}}^2}
\newcommand{\forcetextnewline}{\\}
\newcommand{\Sigmachi}[2]{\Sigma^{#1}_{_{#2}}}
\newcommand{\vpmpm}[1]{v_{_{#1}}}

\title{Higher weight on $GL(3)$, I:\\The Eisenstein series.}
\author{Jack Buttcane}
\date{10 January 2017}
\address{Mathematics Department, 244 Mathematics Building, Buffalo, NY 14260, USA}
\email{buttcane@buffalo.edu}

\begin{document}

\begin{abstract}
The purpose of this paper is to collect and make explicit the results of Langlands \cite{Langlands02}, Bump \cite{Bump01}, Miyazaki \cite{Miya01} and Manabe, Ishii and Oda \cite{ManIshOda} for the $GL(3)$ Eisenstein series and Whittaker functions which are non-trivial on $SO(3,\R)$.
The final goal for the series of papers is a complete and completely explicit spectral expansion for $L^2(SL(3,\Z)\backslash SL(3,\R))$ in the style of Duke, Friedlander and Iwaniec's paper \cite{DFI01}.
We derive a number of new results on the Whittaker functions and Eisenstein series, and give new, concrete proofs of the functional equations and spectral expansion in place of the general constructions of Langlands.
\end{abstract}


\maketitle

\section{Introduction}
In a series of papers \cite{Sel01,Sel02,Sel03}, Selberg gave an analysis of the $L^2$-spaces on compact quotients $\Gamma\backslash G/K$ with $\Gamma$ discrete and $K$ compact, in terms of the eigenfunctions of the natural translation-invariant differential operators, subject to some assumptions on the group $G$.
Langlands \cite{Langlands01,Langlands02} then gave a vast generalization to the case of finite-volume $\Gamma\backslash G$ where $G$ is a connected reductive Lie group, subject to an additional (widely satisfied) technical assumption that $G$ has a finite basis of parabolic subgroups.
Langlands' spectral expansion is a tempting resource for analytic number theorists, but the level of abstraction makes it nearly impossible to apply.

The key component of the spectral expansion of Langlands is a basis of the continuous spectrum consisting of generalized Eisenstein series; one must show the meromorphic continuation and functional equations of these objects, and his approach was to prove the meromorphic continuation directly and show that the functional equations of the constant terms, i.e. the zeroth Fourier coefficients, must hold for the series as a whole.
We may approach the problem instead by considering the full Fourier expansion, a la Piatetski-Shapiro \cite{PS01} and Shalika \cite{Shal01}, in terms of generalized Whittaker functions.
The analytic continuation and functional equations of these Whittaker functions give the necessary results for the Eisenstein series.
This approach was initiated by Jacquet \cite{Jac01}.

In the specific case $\Gamma=SL(3,\Z)$ and $G=PSL(3,\R)$, there are two distinct types of Eisenstein series: one type attached to the minimal parabolic and one attached to the maximal parabolic.
Bump \cite{Bump01} (and independently Vinogradov and Takhtadzhyan \cite{VinTaht}) has given a very explicit computation of the Fourier-Whittaker expansion of the minimal parabolic Eisenstein series, and following Jacquet's approach, its meromorphic continuation and functional equations.
In the process he gave an explicit form for the generalized Whittaker function, and identified its Mellin transform (an idea which Stade \cite{Stade01} has generalized to $GL(n)$).
His results apply to the spectral expansion in the so-called ``spherical'' case, i.e. on $L^2(\Gamma\backslash G/K)$ with $K=SO(3,\R)$.
Manabe, Ishii and Oda \cite{ManIshOda}, as well as Miyazaki \cite{Miya01} have computed the Mellin transform of the Whittaker functions in the minimal weight case, which are the spherical case plus those attached to the 3-dimensional representation of $SO(3,\R)$.
Miyazaki's results are actually slightly more general, but not in a way that is useful to the current discussion.

The Fourier-Whittaker expansion, and hence the meromorphic continuation and functional equations, of the maximal parabolic Eisenstein series in the sense of Piatetski-Shapiro and Shalika can, in principle, be found in Shahidi's book \cite{Shah01}, and by Imai and Terras \cite{ImaiTerras} in an entirely different sense.
In practice, we prefer the Fourier expansion of Piatetski-Shapiro and Shalika over that of Imai and Terras, but Shahidi's book is insufficiently explicit with respect to the various measures, and does not consider the ramified places (e.g. for non-spherical forms).
Miyazaki \cite{Miya01} gives a computation of the Fourier coefficients of the maximal parabolic Eisenstein series; we will give a slightly different derivation in this paper.

Finally, in a paper on Artin $L$-functions, Duke, Friedlander and Iwaniec \cite{DFI01}, have given a fully explicit basis of $GL(2)$ Maass forms at each weight $k \in \Z$.
They accept as input the Hecke eigenvalues of spherical Maass forms and holomorphic modular forms and give the basis in terms of the Fourier-Whittaker expansion.

Our goal is to specialize the results of Langlands to the case $\Gamma=SL(3,\Z)$, $G=PSL(3,\R)$, but in the more general, non-spherical case, i.e. on $L^2(\Gamma\backslash G)$; we take a path intermediate to those of Langlands and Jacquet.
As we are strongly interested in applying this to Kuznetsov-type formulae, we require the Fourier expansions of all relevant Maass forms.
The larger goal for this paper and its successor is to replicate the theorem of Duke, Friedlander and Iwaniec for $GL(3)$ Maass forms.
In fact, the construction and analysis of the maximal parabolic Eisenstein series will use their theorem as input.

The $L^2$-space splits into a direct sum over the representations of $K$; up to isomorphism, these are given by the Wigner $\WigDName$-matrices $\WigDMat{d}:K \to SO(2d+1,\C)$ for each integer $d \ge 0$, which we call the weight.
It is common to collect scalar-valued functions on $\Gamma\backslash G$ whose $K$-part lies in the span of the matrix coefficients of some $\WigDMat{d}$ into row vector valued functions which transform as $f(\gamma g k) = f(g)\WigDMat{d}(k)$, for $\gamma \in \Gamma$, $g\in G$, $k\in K$.
The vector-valued Eisenstein series and Whittaker functions can be further collected into matrix-valued forms.

In section \ref{sect:WhittFuncs}, for each $d$, we construct a matrix-valued Whittaker function $W^d(g,\mu,\psi)$ with spectral parameters $\mu$, character $\psi$ and $K$-type $\WigDMat{d}$ (see \eqref{eq:LEWhittDef}).
We compute all of its degenerate forms explicitly in terms of the classical Whittaker function, and for the non-degenerate form, we compute its Mellin transform as a sum of one-dimensional Mellin-Barnes integrals.
We determine the poles of the Mellin transform, generalizing in the weight direction a project of Stade \cite{Stade01} for the $GL(n)$ Whittaker functions, and note it has exponential decay in tube domains.
This gives us the analytic continuation of the Whittaker function, and we give some basic bounds on the Whittaker function which show the rapid decay in the $y$ components.
We also compute the functional equations and briefly prove the absolute convergence of the Fourier expansion of Piatetski-Shapiro and Shalika.

In section \ref{sect:MinPara}, we construct a matrix-valued minimal parabolic Eisenstein series $E^d(g,\mu)$ with spectral parameters $\mu$ and $K$-type $\WigDMat{d}$ (see \eqref{eq:MinParaEisenDef}).
Relying on the computations of Bump \cite{Bump01}, we give its full, explicit Fourier-Whittaker expansion.
We obtain an initial meromorphic continuation from the Fourier expansion, but prove the functional equations from the functional equations of all of the constant terms, in a modification of Langlands' approach.

In section \ref{sect:MaxPara}, we construct matrix-valued maximal parabolic Eisenstein series $E^d(g,\Phi,\mu_1)$ with spectral parameter $\mu_1$ and $K$-type $\WigDMat{d}$ attached to the $GL(2)$ Maass form $\Phi$ (see \eqref{eq:MaxParaDef}).
Relying on the computations of Miyazaki \cite{Miya01}, we give their Fourier-Whittaker expansions, and give their analytic continuation and functional equations.

Lastly, in section \ref{sect:ResSpec}, we consider the poles of the minimal parabolic Eisenstein series, and show that their contribution to the spectral expansion can be expressed in terms of the constant function and a maximal parabolic Eisenstein series attached to the $GL(2)$ constant function.

Collectively, this leads us to the following:
For each integer $d \ge 0$, we put
\[ \mathcal{S}_2^d = \set{\sqrt{\tfrac{6}{\pi}}}\cup \wtilde{\mathcal{S}}_0 \bigcup_{0 < k \le d, \, 2|k} \wtilde{\mathcal{S}}_k, \]
where $\wtilde{\mathcal{S}}_0$ is an orthonormal basis of $SL(2,\Z)$ Hecke-Maass cusp forms\footnote{The $L^2$-normalized constant function on $SL(2)$ is usually $\sqrt{\frac{3}{\pi}}$, the spare 2 here is the size of the group $\wbar{V}_2$ in section \ref{sect:MaxParaBruhat}.}, and $\wtilde{\mathcal{S}}_{k}$ is an orthonormal basis of holomorphic Hecke modular forms of weight $k$.
The primary goal of the paper is to make explicit the spectral expansion of Langlands, which becomes
\begin{thm}
\label{thm:ContResSpectralExpand}
	For $f:\Gamma\backslash G \to \C$ smooth and compactly supported, the function $f-f_0-f_1-f_2$ with
	\begin{align*}
		f_0(g) =& \frac{1}{24} \sum_{d=0}^\infty \frac{(2d+1)}{(2\pi i)^2} \int_{\Re(\mu) = 0} \Tr\Bigl(E^d(g, \mu) \int_{\Gamma\backslash G} f(g') \wbar{\trans{E^d(g', \mu)}} dg' \Bigr) \, d\mu,
	\end{align*}
	\begin{align*}
		f_1(g) =& \sum_{d=0}^\infty \sum_{\Phi\in\mathcal{S}_2^d} \frac{(2d+1)}{2\pi i} \int_{\Re(\mu_1) = 0} \Tr\biggl(E^d(g, \Phi, \mu_1) \int_{\Gamma\backslash G} f\paren{g'} \wbar{\trans{E^d(g', \Phi, \mu_1)}} \, dg' \biggr) d\mu_1,
	\end{align*}
	\begin{align*}
		f_2 =& \frac{1}{4/\zeta(3)} \int_{\Gamma\backslash G} f(g) dg,
	\end{align*}
	is square-integrable and cuspidal.
\end{thm}
We say a function $f$ is cuspidal when its degenerate Fourier coefficients are all zero, that is, when the integrals
\[ \int_{U_{w_4}(\Z)\backslash U_{w_4}(\R)} f(ug) du, \qquad \int_{U_{w_5}(\Z)\backslash U_{w_5}(\R)} f(ug) du, \]
on the unipotent subgroups $U_w(\R)\subset G$, defined in section \ref{sect:GSC} below, are both zero.
Such an $f$ will fall in the span of the discrete spectrum, a.k.a. the Maass cusp forms.
Describing those forms will require knowledge of the raising and lowering operators on $GL(3)$, which will be the focus of the next paper.

A note on absolute convergence:
The first step in the spectral expansion is to expand $f$ over $K$-types $\WigDMat{d}$.
Since the entries of these matrices are eigenfunctions of the Laplacian on $SO(3,\R)$ (with supremum norm 1), we may assume the $d$ sum is finite, up to a negligible error (depending on the derivatives of $f$), by the usual harmonic analysis trickery.
Then all of the matrix-valued Eisenstein series are eigenfunctions of the Laplacian on $G$, so we may again assume their spectral parameters are bounded, applying the trivial bound on the supremum norm given in Proposition \ref{eq:FourierExpBound}.

It is interesting to determine the multiplicities of the (row) vector-valued Eisenstein series, and this can be done using the asymptotics of the Whittaker functions and explicit forms of the Fourier-Whittaker expansions.
For the minimal parabolic Eisenstein series and a maximal parabolic Eisenstein series attached to a spherical $GL(2)$ Maass form, the asymptotics of section \ref{sect:WhittAsymps} are given by invertible matrices, so the multiplicities can be read directly from the rank of the $\Sigmachi{d}{+\pm}$ matrix (see section \ref{sect:VGroup}) in the Fourier-Whittaker expansion:
The multiplicities are
\[ \piecewise{\frac{d}{2}+1&\If d \text{ is even},\\ \frac{d-1}{2}&\If d \text{ is odd},} \]
for the minimal parabolic or maximal parabolic attached to an even $GL(2)$ Maass form, and
\[ \piecewise{\frac{d}{2}&\If d \text{ is even},\\ \frac{d+1}{2}&\If d \text{ is odd},} \]
for a maximal parabolic attached to an odd $GL(2)$ Maass form.
A maximal parabolic Eisenstein series attached to an even weight $\kappa > 0$ modular form (as we have constructed it) is the same as for an even Maass form, except the middle $2\kappa-1$ rows are deleted, so the sequence becomes
\[ \piecewise{\frac{d-\kappa}{2}+1&\If d\ge\kappa \text{ is even},\\ \frac{d-\kappa+1}{2}&\If d>\kappa \text{ is odd}.} \]
The use of matrix-valued forms in this paper makes this computation a novelty, at least in terms of the spectral expansion.

\section{Background}
\subsection{Groups, spaces and characters}
\label{sect:GSC}
Let $G=PSL(3,\R) = GL(3,\R)/\R^\times$ and $\Gamma=SL(3,\Z)$.
In section \ref{sect:Iwasawa}, we will give explicitly the Iwasawa decomposition of $G=U(\R) Y^+ K$.
The groups involved in the decomposition are the orthogonal matrices $K=SO(3,\R)$, the unipotent matrices
\[ U(R) = \set{\Matrix{1&x_2&x_3\\&1&x_1\\&&1} \setdiv x_i\in R}, \qquad R \in \set{\R,\Q,\Z}, \]
(the missing entries to be interpreted as zero) and the diagonal matrices
\[ Y^+ = \set{\Matrix{y_1 y_2\\& y_1\\&&1} \in GL(3,\R) \setdiv y_1,y_2 > 0}. \]
The relationship between the indices of the $x_i$ and their location within the $x$-matrix (in $U$) will be fixed throughout the paper, similarly for the $y$ matrices and their coordinates.
The measure on the space $U(\R)$ is simply $dx := dx_1 \, dx_2 \, dx_3$, and the measure on $Y^+$ is
\[ dy := \frac{dy_1 \, dy_2}{(y_1 y_2)^3}. \]
The measure on $G$ is $dg :=dx \, dy \, dk$, where $dk$ is the Haar probability measure on $K$ described in section \ref{sect:KL2}.
We do not distinguish notationally between the matrix $y=\SmallMatrix{y_1 y_2\\& y_1\\&&1}$ and the pair $y=(y_1, y_2)$.
This will not cause a problem, as the multiplication is the same in both realizations.

The intersection of the diagonal and orthogonal matrices will frequently arise in computations; this is the group $V$ containing the four matrices
\[ \vpmpm{++}= I, \vpmpm{+-} = \Matrix{1\\&-1\\&&-1}, \vpmpm{-+} = \Matrix{-1\\&-1\\&&1}, \vpmpm{--} = \Matrix{-1\\&1\\&&-1}. \]

Characters of $U(\R)$ are given by
\[ \psi_m(x) = \psi_{m_1,m_2}(x) = \e{m_1 x_1+m_2 x_2}, \qquad \e{t} = e^{2\pi i t}, \]
where $m\in\R^2$; the symbol $\psi$ will generally denote such a character.
The $V$ group acts on these characters by $\psi_m^v(x) = \psi_m(vxv)$.
This gives the action $\psi_{m^v}(x)=\psi_m^v(x)$ on $\R^2$,
\[ m^{\vpmpm{++}} = m, \qquad m^{\vpmpm{-+}} = (m_1,-m_2), \qquad m^{\vpmpm{+-}} = (-m_1,m_2), \qquad m^{\vpmpm{--}} = (-m_1,-m_2). \]

Characters of $Y^+$ are given by the power function on $3\times 3$ diagonal matrices, defined by
\[ p_\mu\Matrix{a_1\\&a_2\\&&a_3} = \abs{a_1}^{\mu_1} \abs{a_2}^{\mu_2} \abs{a_3}^{\mu_3}, \]
where $\mu\in\C^3$.
We assume $\mu_1+\mu_2+\mu_3=0$ so this is defined modulo $\R^\times$, renormalize by $\rho=(1,0,-1)$, and extend by the Iwasawa decomposition
\[ p_{\rho+\mu}\paren{r x y k} = y_1^{1-\mu_3} y_2^{1+\mu_1}, \qquad r\in\R^\times,x\in U(\R),y\in Y^+,k\in K. \]
When we integrate over the space of such $\mu$, we will use the standard measure
\[ d\mu = d\mu_1 \, d\mu_2 = d\mu_2 \, d\mu_3 = d\mu_1 \, d\mu_3. \]

We will also make reference to the Bruhat decomposition of $\Gamma \subset G$.
This decomposition has the form $G=U(\R) Y^+ V W U(\R)$, where $W$ is the Weyl group of $G$ containining the six matrices
\begin{equation*}
	\begin{array}{rclcrclcrcl}
		I &=& \Matrix{1\\&1\\&&1}, && w_2 &=& -\Matrix{&1\\1\\&&1}, && w_3 &=& -\Matrix{1\\&&1\\&1}, \\
		w_4 &=& \Matrix{&1\\&&1\\1}, && w_5 &=& \Matrix{&&1\\1\\&1}, && w_l &=& -\Matrix{&&1\\&1\\1}.
	\end{array}
\end{equation*}
Note that this matches Bump's book \cite{Bump01}, but has $w_2$ and $w_3$ interchanged compared to Goldfeld's book \cite{Gold01}.
When taking the Bruhat decomposition of an element $\gamma\in\Gamma$, we have $\gamma=bcvwb'$ with $b,b'\in U(\Q)$, $v\in V$, $w\in W$ and $c$ of the form
\[ \Matrix{\frac{1}{c_2}\\&\frac{c_2}{c_1}\\&&c_1}, \qquad c_1, c_2\in \N. \]

The space $U$ decomposes into subspaces according to the action of the Weyl group, and we set $\wbar{U}_w = (w^{-1} \, \trans{U} \, w) \cap U$ for $w\in W$.
These are
\begin{align}
\label{eq:ExplicitUw}
\wbar{U}_I = \set{I}, \qquad \wbar{U}_{w_2} = \set{\Matrix{1&x_2&0\\&1&0\\&&1}}, \qquad \wbar{U}_{w_3} = \set{\Matrix{1&0&0\\&1&x_1\\&&1}} \\
\wbar{U}_{w_4} = \set{\Matrix{1&x_2&x_3\\&1&0\\&&1}}, \qquad \wbar{U}_{w_5} = \set{\Matrix{1&0&x_3\\&1&x_1\\&&1}}, \qquad \wbar{U}_{w_l} = U. \nonumber
\end{align}
The complementary spaces are $U_w = (w^{-1} \, U \, w) \cap U$.
The Bruhat decomposition of $\gamma \in \Gamma$ above becomes unique if we enforce $b'\in \wbar{U}_w(\Q)$.

The Weyl group also defines a right action on the coordinates of $\mu$ by permuting the entries of a diagonal matrix,
\[ p_{\mu^w}(a)=p_\mu\paren{waw^{-1}}. \]
When applying subscripts, we will use the convention $\mu^w_1 = (\mu^w)_1$.
Because it arises so frequently, we give a table of this action:
\begin{align*}
	\mu^I =& \paren{\mu_1,\mu_2,\mu_3}, & \mu^{w_2} =& \paren{\mu_2,\mu_1,\mu_3}, & \mu^{w_3} =& \paren{\mu_1,\mu_3,\mu_2}, \\
	\mu^{w_4} =& \paren{\mu_3,\mu_1,\mu_2}, & \mu^{w_5} =& \paren{\mu_2,\mu_3,\mu_1}, & \mu^{w_l} =& \paren{\mu_3,\mu_2,\mu_1}.
\end{align*}

The Fourier coefficients (or Hecke eigenvalues) of the $GL(3)$ Eisenstein series are most easily described in terms of the Schur polynomials
\begin{align}
\label{eq:SchurDef}
	S_{n_1,n_2}(\alpha,\beta) :=& \frac{\det \Matrix{1&\beta^{n_1+n_2+2}&\alpha^{n_1+n_2+2}\\1&\beta^{n_1+1}&\alpha^{n_1+1}\\1&1&1}}{\det \Matrix{1&\beta^2&\alpha^2\\1&\beta&\alpha\\1&1&1}}.
\end{align}

Lastly, we will require the Pl\"ucker coordinates $A_i, B_i, C_i$ on elements $\gamma\in\Gamma$ defined by
\[ \gamma=\Matrix{*&*&*\\d&e&f\\A_1&B_1&C_1}\Rightarrow \begin{array}{c} A_2=B_1d-A_1e,\\B_2=A_1f-C_1d,\\C_2=C_1e-B_1f\end{array} \]
satisfying $A_1 C_2+B_1 B_2+C_1 A_2=0$, and
\[ (A_1, B_1, C_1)=(A_2,B_2,C_2)=1. \]
These six integers uniquely determine cosets in $U(\Z)\backslash\Gamma$.

\subsection{Representations of $SO(3,\R)$}
\label{sect:KReps}
Up to isomorphism, the representations of $K$ are given by the Wigner $\WigDName$-matrices, in the physics terminology.
These are frequently deduced by the isomorphism of the Lie algebras of $SO(3)$ and $SU(2)$, but we will not be concerned with this construction.
A good reference for the representation-theoretic description of Wigner $\WigDName$-matrices is \cite{Kn01} (this can also be found in \cite{ManIshOda} and \cite{Miya01}), and for the analytic description, \cite{BL01}.

Say $k = k(\alpha,\beta,\gamma) \in K$ is described in terms of the $Z$-$Y$-$Z$ Euler angles,
\begin{align}
\label{eq:Kparam}
	k =& \Matrix{\cos\alpha&-\sin\alpha&0\\ \sin\alpha&\cos\alpha&0\\ 0&0&1} \Matrix{\cos\beta&0&\sin\beta\\0&1&0\\-\sin\beta&0&\cos\beta} \Matrix{\cos\gamma&-\sin\gamma&0\\ \sin\gamma&\cos\gamma&0\\0&0&1},
\end{align}
with $0 \le \alpha,\gamma < 2\pi$, $0 \le \beta \le \pi$, then we write the Wigner $\WigDName$-matrix of dimension $2d+1$ with entry at the $(d+1+m')$-th row and $(d+1+m)$-th column
\begin{align}
\label{eq:FirstWigD}
	\WigD{d}{m'}{m}(k) = e^{-im'\alpha} \Wigd{d}{m'}{m}(\cos\beta) e^{-im\gamma}.
\end{align}
To avoid the ever-present $\arccos$, we have departed slightly from the standard notation $\Wigd{d}{m'}{m}(\beta)$, in which $\Wigd{d}{m'}{m}$ is a trigonometric polynomial in $\beta/2$ called the Wigner $\WigdName$-polynomial.
The matrix and associated vectors are generally indexed from the center entry so that, e.g. the matrix indices match the indices of the characters of the two outer copies of $SO(2,\R)$ in the $Z$-$Y$-$Z$ Euler angle description.

It is frequently useful to rewrite the parameterization $k(\alpha,\beta,\gamma)$ of $K$ given in \eqref{eq:Kparam} as 
\begin{align}
\label{eq:Kaltparam}
	\tilde{k}(\exp i\alpha, \exp i\beta, \exp i\gamma) = k(\alpha,\beta,\gamma),
\end{align}
keeping in mind the relevant ranges of the angles.
In this notation, the inverse mapping becomes
\begin{align*}
	\Matrix{g&h&j\\d&e&f\\a&b&c} = \tildek{\frac{j+if}{\sqrt{j^2+f^2}},\frac{c+i\sqrt{a^2+b^2}}{\sqrt{a^2+b^2+c^2}},\frac{-a+ib}{\sqrt{a^2+b^2}}},
\end{align*}
unless $a=b=0$.
In that case, necessarily $j=f=0$, and we may take
\begin{align*}
	\Matrix{g&h&0\\d&e&0\\0&0&\pm 1} = \tildek{\frac{e\pm id}{\sqrt{e^2+d^2}},\pm 1,1}.
\end{align*}

The Wigner $\WigDName$-matrices give representations
\[ \WigDMat{d}:K\to SL(2d+1,\C), \quad \WigDMat{d}(kk')=\WigDMat{d}(k) \WigDMat{d}(k'), \quad \WigDMat{d}\paren{k^{-1}} = \WigDMat{d}(k)^{-1} = \wbar{\trans{\WigDMat{d}(k)}}, \]
and the collection of such representations for all integral $d \ge 0$ exhausts the isomorphism classes for representations of $K$.
A particularly useful trick for avoiding the Wigner-$\WigdName$ polynomials, when one knows the Euler angles, is to write $k(\alpha,\beta,\gamma)$ as
\[ k(\alpha,0,0) k(0,\beta,0) k(0,0,\gamma) = k(\alpha,0,0) w_3 k(-\beta,0,0) w_3 k(\gamma,0,0), \]
so that
\begin{align*}
	\WigDMat{d}(k(\alpha,\beta,\gamma))=\WigDMat{d}(k(\alpha,0,0)) \WigDMat{d}(w_3) \WigDMat{d}(k(-\beta,0,0)) \WigDMat{d}(w_3) \WigDMat{d}(k(\gamma,0,0)).
\end{align*}
Since it arises often, we will abbreviate $\Dtildek{d}{s} = \WigDMat{d}\paren{\tildek{s,1,1}}$.
Then from \eqref{eq:FirstWigD}, this is
\begin{align}
\label{eq:tildeDdiag}
	\Dtildek{d}{s} = \diag\paren{s^d, \ldots, 1, \ldots, s^{-d}},
\end{align}
i.e. $\WigD{d}{m}{m}(\tildek{s,1,1})=s^{-m}$.
In this notation, the trick becomes
\begin{align}
\label{eq:w3Trick}
	\WigDMat{d}(\tildek{s,t,u})=\Dtildek{d}{s} \WigDMat{d}(w_3) \Dtildek{d}{\wbar{t}} \WigDMat{d}(w_3) \Dtildek{d}{u}.
\end{align}
This immediately implies a representation of the Wigner $\WigdName$-polynomials,
\begin{align}
\label{eq:WigdExplicit}
	\Wigd{d}{m'}{m}(\Re(t)) = \WigD{d}{m'}{m}(\tildek{1,t,1})=\sum_{j=-d}^d \WigD{d}{m'}{j}(w_3) \WigD{d}{j}{m}(w_3) t^j, \quad \abs{t}=1, 0 \le \arg(t) \le \pi,
\end{align}
and a trivial corollary is
\[ \pderv[n_1]{}{s} \pderv[n_2]{}{t} \pderv[n_3]{}{u} \WigD{d}{m'}{m}(\tildek{s,t,u}) \ll d^{1+n_1+n_2+n_3}, \]
since $\abs{\WigD{d}{m'}{m}(w_3)} \le 1$.
The extra $d^1$ here is not optimal.
We will also occasionally need the dual representation
\begin{align}
\label{eq:dualWignerD}
	\dualWigDMat{d}(k) = \WigDMat{d}(w_l k w_l).
\end{align}

\subsubsection{The $L^2$-space on $K$}
\label{sect:KL2}
The Haar probability measure on $K$, in terms of the $Z$-$Y$-$Z$ Euler angles, is given by
\[ dk = \frac{1}{8\pi^2} \sin\beta \, d\alpha \, d\beta \, d\gamma, \]
and by the Peter-Weyl theorem, the entries of the Wigner $\WigDName$-matrices give a basis for the $L^2$-space:
If $f\in L^2(K)$, then
\begin{align}
\label{eq:Kexpand}
	f(k) =& \sum_{d\ge0} (2d+1) \sum_{\abs{m'},\abs{m} \le d} \what{f}(d)_{m,m'} \WigD{d}{m'}{m}(k) \\
	=& \sum_{d\ge0} (2d+1) \sum_{\abs{m} \le d} \what{f}(d)_{m,\cdot} \WigD{d}{\cdot}{m}(k) \nonumber \\
	=& \sum_{d\ge0} (2d+1) \Tr\paren{\what{f}(d) \WigDMat{d}(k)}, \nonumber
\end{align}
where
\[ \what{f}(d) = \int_K f(k) \wbar{\trans{\WigDMat{d}(k)}} dk, \]
again indexing from the center element.

\subsubsection{The $V$ group}
\label{sect:VGroup}
The $V$ matrices may be written (not uniquely) as
\begin{align*}
\begin{array}{rlcrlcrlcrl}
	k(0,0,0) =& I, && k(\pi,0,0) =& \vpmpm{-+}, && k(0,\pi,0) =& \vpmpm{--}, && k(\pi,\pi,0) =& \vpmpm{+-}.
\end{array}
\end{align*}

When $\beta=0$, the Wigner-$\WigdName$ polynomial is zero unless $m=m'$, and in that case
\[ \Wigd{d}{m'}{m'}(\cos 0) = 1. \]
When $\beta=\pi$, the Wigner-$\WigdName$ polynomial is zero unless $m=-m'$, and in that case
\[ \Wigd{d}{m'}{-m'}(\cos \pi) = (-1)^{d-m'}. \]
Putting these together,
\begin{align*}
	\WigDMat{d}(I) =& I, & \WigDMat{d}(\vpmpm{-+}) =& \Matrix{(-1)^d\\&\ddots\\&&(-1)^0\\&&&\ddots\\&&&&(-1)^d}, \\
	\WigDMat{d}(\vpmpm{+-}) =& (-1)^d \Matrix{&&&&1\\&&&\revdots\\&&1\\&\revdots\\1}, & \WigDMat{d}(\vpmpm{--}) =& \Matrix{&&&&(-1)^0\\&&&\revdots\\&&(-1)^d\\&\revdots\\(-1)^0}. \\
\end{align*}

A character $\chi$ of $V$ is determined by its values on $\vpmpm{+-}$ and $\vpmpm{-+}$.
The group of all such characters is then
\[ \set{\chi_{++}, \chi_{+-}, \chi_{-+}, \chi_{--}}, \]
where $\chi_{\pm\cdot}(\vpmpm{-+})=\pm1$ and $\chi_{\cdot\pm}(\vpmpm{+-})=\pm1$.
Define
\[ \Sigma^d_\chi := \frac{1}{\abs{V}} \sum_{v\in V} \chi(v) \WigDMat{d}(v), \]
and we use the short-hand $\Sigmachi{d}{\pm\pm} = \Sigmachi{d}{\chi_{\pm\pm}}$.
These matrices are projection operators, and hence satisfy the handy property $\Sigma^d_\chi = \Sigma^d_\chi\Sigma^d_\chi$.

The matrix $\Sigmachi{d}{++}$ occurs in the Fourier-Whittaker expansion of the minimal parabolic Eisenstein series.
For the maximal parabolic Eisenstein series, we will encounter a parity, which mimics a character of the $V$ group, so these will have either $\Sigmachi{d}{++}$ or $\Sigmachi{d}{+-}$ in their Fourier-Whittaker expansion.

These sums take the values
\begin{align}
\label{eq:Sigmadplusplus}
	\Sigmachi{2d}{++} = \Matrix{\ddots &&&&&& \revdots \\ &\frac{1}{2}&&&&\frac{1}{2} \\ &&0&&0 \\ &&&1 \\ &&0&&0 \\ &\frac{1}{2}&&&&\frac{1}{2} \\ \revdots &&&&&& \ddots}, \qquad \Sigmachi{2d+1}{++} = \Matrix{\ddots &&&&&& \revdots \\ &\frac{1}{2}&&&&-\frac{1}{2} \\ &&0&&0 \\ &&&0 \\ &&0&&0 \\ &-\frac{1}{2}&&&&\frac{1}{2} \\ \revdots &&&&&& \ddots},
\end{align}
\begin{align}
\label{eq:Sigmadplusminus}
	\Sigmachi{2d}{+-} = \Matrix{\ddots &&&&&& \revdots \\ &\frac{1}{2}&&&&-\frac{1}{2} \\ &&0&&0 \\ &&&0 \\ &&0&&0 \\ &-\frac{1}{2}&&&&\frac{1}{2} \\ \revdots &&&&&& \ddots}, \qquad \Sigmachi{2d+1}{+-} = \Matrix{\ddots &&&&&& \revdots \\ &\frac{1}{2}&&&&\frac{1}{2} \\ &&0&&0 \\ &&&1 \\ &&0&&0 \\ &\frac{1}{2}&&&&\frac{1}{2} \\ \revdots &&&&&& \ddots},
\end{align}
\begin{align}
\label{eq:Sigmadminusplus}
	\Sigmachi{2d}{-+} = \Matrix{\ddots &&&&&& \revdots \\ &0&&&&0 \\ &&\frac{1}{2}&&\frac{1}{2}\\ &&&0 \\ &&\frac{1}{2}&&\frac{1}{2} \\ &0&&&&0 \\ \revdots &&&&&& \ddots}, \qquad \Sigmachi{2d+1}{-+} = \Matrix{\ddots &&&&&& \revdots \\ &0&&&&0 \\ &&\frac{1}{2}&&-\frac{1}{2}\\ &&&0 \\ &&-\frac{1}{2}&&\frac{1}{2} \\ &0&&&&0 \\ \revdots &&&&&& \ddots},
\end{align}
\begin{align}
\label{eq:Sigmadminusminus}
	\Sigmachi{2d}{--} = \Matrix{\ddots &&&&&& \revdots \\ &0&&&&0 \\ &&\frac{1}{2}&&-\frac{1}{2}\\ &&&0 \\ &&-\frac{1}{2}&&\frac{1}{2} \\ &0&&&&0 \\ \revdots &&&&&& \ddots}, \qquad \Sigmachi{2d+1}{--} = \Matrix{\ddots &&&&&& \revdots \\ &0&&&&0 \\ &&\frac{1}{2}&&\frac{1}{2}\\ &&&0 \\ &&\frac{1}{2}&&\frac{1}{2} \\ &0&&&&0 \\ \revdots &&&&&& \ddots}.
\end{align}

The two relations that will appear are
\begin{align}
\label{eq:SigmadRels}
	\WigDMat{d}(w_l) \Sigmachi{d}{-+} \WigDMat{d}(w_l) = \Sigmachi{d}{+-}, \qquad \WigDMat{d}(w_3) \Sigmachi{d}{-+} \WigDMat{d}(w_3) =\Sigmachi{d}{-+}.
\end{align}

\subsubsection{The Weyl group}
The Weyl elements may be written in $Z$-$Y$-$Z$ coordinates as
\begin{align*}
\begin{array}{ccccc}
	k(\tfrac{\pi}{2},\pi,0) = w_2, & k(\tfrac{3\pi}{2},\tfrac{\pi}{2},\tfrac{3\pi}{2}) = w_3, & k(\tfrac{\pi}{2},\tfrac{\pi}{2},\pi) = w_4, & k(0,\tfrac{\pi}{2},\tfrac{\pi}{2}) = w_5, & k(\pi,\tfrac{\pi}{2},0) = w_l.
\end{array}
\end{align*}

\subsection{Two classical integrals}
\subsubsection{The classical Whittaker function}
For $m\in \Z$, $\Re(u)>0$ and $y\in\R$, consider the absolutely convergent integral
\begin{align}
\label{eq:WintOld}
	\mathcal{W}_m(y, u) =& \int_{-\infty}^\infty \paren{1+x^2}^{-\frac{1+u}{2}} \paren{\frac{1+ix}{\sqrt{1+x^2}}}^{-m} \e{-y x} dx.
\end{align}
This is the integral $I_s(n,y)$ from \cite[Section 7]{DFI01}, and it may be evaluated explicitly using \cite[3.384.9]{GradRyzh} for $y \ne 0$ and \cite[8.381.1]{GradRyzh} for $y=0$, giving
\begin{align}
\label{eq:Weval}
	\mathcal{W}_m(y, u) =& (-1)^m \frac{(\pi\abs{y})^{\frac{1+u}{2}}}{\abs{y} \Gamma\paren{\frac{1-\varepsilon m+u}{2}}} W_{-\frac{\varepsilon m}{2}, \frac{u}{2}}(4\pi\abs{y}), \qquad y \ne 0 \\
\label{eq:Beval}
	\mathcal{W}_m(0,u) =& \frac{2^{1-u} \pi \Gamma\paren{u}}{\Gamma\paren{\frac{1+u+m}{2}} \Gamma\paren{\frac{1+u-m}{2}}} = \pi^{\frac{1}{2}} \frac{\Gamma\paren{\frac{u}{2}} \Gamma\paren{\frac{1+\abs{m}-u}{2}}}{\Gamma\paren{\frac{1-u}{2}} \Gamma\paren{\frac{1+\abs{m}+u}{2}}} \frac{\sin\paren{\pi \frac{1+u-\abs{m}}{2}}}{\sin\paren{\pi\frac{1+u}{2}}},
\end{align}
where $\varepsilon = \sgn(y)$ and $W_{\alpha,\beta}(y)$ is the classical Whittaker function.
Note that one should apply the formulae in Gradshteyn \& Ryshik for $u$ real, and the principal value of the power function, and extend to $\Re(u) > 0$ by analytic continuation.
We collect these terms into the matrix $\mathcal{W}^d(y,u)$ having diagonal entries $\mathcal{W}^d_{m,m}(y,u) = \mathcal{W}_m(y,u)$, indexing from the center of the matrix, as usual.

Note that the functional equation of the Whittaker function $W_{\alpha,-\beta}(y)=W_{\alpha,\beta}(y)$ implies
\begin{align}
\label{eq:ClassWhittFE}
	\mathcal{W}^d(y,-u) =& (\pi\abs{y})^{-u} \Gamma_\mathcal{W}^d(u,\varepsilon) \mathcal{W}^d(y,u),
\end{align}
where $\Gamma_\mathcal{W}^d(u,\varepsilon)$ is the diagonal matrix with entries
\begin{align}
\label{eq:WhittGammas}
	\Gamma_{\mathcal{W},m,m}^d(u,\varepsilon) = \frac{\Gamma\paren{\frac{1-\varepsilon m+u}{2}}}{\Gamma\paren{\frac{1-\varepsilon m-u}{2}}} = (-\varepsilon\sgn(m))^m \frac{\Gamma\paren{\frac{1+\abs{m}+u}{2}}}{\Gamma\paren{\frac{1+\abs{m}-u}{2}}}.
\end{align}
The second equality follows by the reflection formula; it will be important that the leading sign disappears for even $m$.
Similarly, we notice on the even entries the $y=0$ case simplifies to
\begin{align}
\label{eq:WevenRows}
	\mathcal{W}_{2m}(0,u) =& (-1)^m \pi^{\frac{1}{2}} \frac{\Gamma\paren{\frac{u}{2}} \Gamma\paren{\frac{1-u}{2}+\abs{m}}}{\Gamma\paren{\frac{1-u}{2}} \Gamma\paren{\frac{1+u}{2}+\abs{m}}}.
\end{align}

Using the fact \eqref{eq:tildeDdiag}, the matrix for the classical integral becomes
\begin{align}
\label{eq:Wint}
	\mathcal{W}^d(y, u) =& \int_{-\infty}^\infty \paren{1+x^2}^{-\frac{1+u}{2}} \Dtildek{d}{\frac{1+ix}{\sqrt{1+x^2}}} \e{-y x} dx.
\end{align}
Notice the commutativity of $SO(2,\R)$ gives commutativity of the matrices
\begin{align}
\label{eq:WDtildek}
	\mathcal{W}^d(y, u) \Dtildek{d}{s} = \Dtildek{d}{s} \mathcal{W}^d(y, u).
\end{align}

Sending $x\mapsto -x$ in the original \eqref{eq:WintOld} gives $\mathcal{W}_{-m}(y,u) = \mathcal{W}_m(-y,u)$, so we also have the symmetry
\begin{align}
\label{eq:Wvpm}
	\mathcal{W}^d(-y, u) = \WigDMat{d}(\vpmpm{+-})\mathcal{W}^d(y, u)\WigDMat{d}(\vpmpm{+-}).
\end{align}
In case $y=0$, this translates to commutivity
\begin{align}
\label{eq:W0v}
	\mathcal{W}^d(0, u)\WigDMat{d}(v) = \WigDMat{d}(v)\mathcal{W}^d(0, u),
\end{align}
for all $v\in V$.

\subsubsection{A generalized beta function}
For $\varepsilon \in \set{\pm 1}$, $m\in\Z$, $\Re(a) > 0$ and $\Re(b) > 0$, consider the integral
\begin{align}
\label{eq:BetaFunDef}
	\mathcal{B}_{\varepsilon, m}(a,b) =& \int_0^\infty x^{a-1} (1+x^2)^{-\frac{b+a}{2}} \paren{\paren{\frac{1+ix}{\sqrt{1+x^2}}}^{-m}+\varepsilon\paren{\frac{1-ix}{\sqrt{1+x^2}}}^{-m}} dx.
\end{align}
Since the quantities in the $m$-th powers have modulus one, we have
\[ \mathcal{B}_{\varepsilon, m}(a,b) = \varepsilon \mathcal{B}_{\varepsilon, -m}(a,b), \]
and we may assume $m \ge 0$.
It is worth pointing out that $\mathcal{B}_{\varepsilon, m}(a,b)$ is essentially the Mellin transform of $\mathcal{W}_m(y, u)$.

By expanding the $m$-th powers, and applying
\begin{align}
\label{eq:ClassBetaFun}
	\int_0^\infty x^a (1+x^2)^b dx = \frac{1}{2} B\paren{\frac{1+a}{2}, \frac{-1-a-2b}{2}},
\end{align}
where $B(u,v)=\frac{\Gamma(u)\Gamma(v)}{\Gamma(u+v)}$ is the Euler beta function, we arrive at
\begin{align}
\label{eq:BexplicitEvalEven}
	\mathcal{B}_{1, m}(a,b) =& \sum_{j=0}^{m/2} \binom{m}{2j} (-1)^j B\paren{\frac{a}{2}+j, \frac{m+b}{2}-j}, \\
\label{eq:BexplicitEvalOdd}
	\mathcal{B}_{-1, m}(a,b) =& i \sum_{j=0}^{(m-1)/2} \binom{m}{2j+1} (-1)^j B\paren{\frac{1+a}{2}+j, \frac{m-1+b}{2}-j},
\end{align}
for $m \ge 0$.

Now if $m \equiv \delta \pmod{2}$ with $\delta\in\set{0,1}$ we note that
\begin{align}
\label{eq:BfactorEven}
	\mathcal{B}_{1, m}(a,b) =& \wtilde{\mathcal{B}}_{1, m}(a,b) \frac{\Gamma\paren{\frac{a}{2}}\Gamma\paren{\frac{\delta+b}{2}}}{\Gamma\paren{\frac{m+a+b}{2}}}, \\
\label{eq:BfactorOdd}
	\mathcal{B}_{-1, m}(a,b) =& \wtilde{\mathcal{B}}_{-1, m}(a,b) \frac{\Gamma\paren{\frac{1+a}{2}}\Gamma\paren{\frac{1-\delta+b}{2}}}{\Gamma\paren{\frac{m+a+b}{2}}},
\end{align}
where $\wtilde{\mathcal{B}}_{\varepsilon, m}(a,b)$ is a polynomial in $a,b$.
Hence we may say $\mathcal{B}_{\varepsilon, m}(a,b)$ has meromorphic continuation to all $a,b\in\C$ with at most simple poles coming from the gamma functions in the numerator and at least simple zeros coming from the gamma function in the denominator, assuming there is no overlap between the sets of poles coming from any pair of the gamma functions.

As with the previous integral, we collect these terms into the matrices $\mathcal{B}^d_\varepsilon(a,b)$ and $\wtilde{\mathcal{B}}_\varepsilon^d(a,b)$ having diagonal entries, e.g. $\mathcal{B}^d_{\varepsilon,m,m}(a,b) = \mathcal{B}^d_{\varepsilon,m}(a,b)$, indexing from the center of the matrix, as usual.
In the matrix notation, this is
\begin{align}
\label{eq:Bint}
	\mathcal{B}^d_\varepsilon(a,b) =& \int_0^\infty x^{a-1} (1+x^2)^{-\frac{b+a}{2}} \paren{\Dtildek{d}{\frac{1+ix}{\sqrt{1+x^2}}}+\varepsilon \Dtildek{d}{\frac{1-ix}{\sqrt{1+x^2}}}} dx.
\end{align}

We will also need a very weak bound showing the exponential decay in $\Im(a)$ and $\Im(b)$.
If $\abs{\Re(a)},\abs{\Re(b)} \le T$, then
\begin{align*}
	\abs{\wtilde{\mathcal{B}}_{1, m}(a,b)} \le& \sum_{j=0}^{m/2} \binom{m}{2j} \paren{\frac{m+T+\abs{\Im(a)}}{2}}^j \paren{\frac{m+T+\abs{\Im(b)}}{2}}^{\floor{m/2}-j} \\
	\le& \paren{\sqrt{\frac{m+T+\abs{\Im(a)}}{2}}+\sqrt{\frac{m+T+\abs{\Im(b)}}{2}}}^m,
\end{align*}
and the same holds for $\wtilde{\mathcal{B}}_{-1, -m}(a,b)$.
For bounded values of the parameters $m$ and $T$, we may write this as
\begin{align}
\label{eq:BtildeBound}
	\abs{\wtilde{\mathcal{B}}_{\varepsilon, m}(a,b)} \ll_{m,T}& \paren{1+\abs{\Im(a)}+\abs{\Im(b)}}^{\abs{m}/2}.
\end{align}
Then Stirling's formula applied to \eqref{eq:BfactorEven} and \eqref{eq:BfactorOdd} gives exponential decay in the region $\abs{\Im(a-b)} > \abs{\Im(a+b)}$.

\subsection{The full Iwasawa decomposition}
\label{sect:Iwasawa}
The Iwasawa decomposition for $g\in GL(3,\R)$ is $g=rxyk$ with $r\in\R^\times$,$x \in U(\R)$, $y\in Y^+$, and $k\in K$.
This can be thought of as an accounting of the Gram-Schmidt procedure.
Say
\begin{align}
\label{eq:FullIwasawa}
	A := \Matrix{g&h&j\\d&e&f\\a&b&c} =& r \Matrix{1&x_2&x_3\\&1&x_1\\&&1} \Matrix{y_1 y_2\\&y_1\\&&1} k,
\end{align}
\begin{align*}
	\xi_1 = (a,b,c), \qquad \xi_2 = (bd-ae, cd-af, ce-bf), \qquad \xi_3 = (bg-ah,cg-aj,ch-bj),
\end{align*}
with $\det A > 0$, then
\begin{align*}
	x_1 = \frac{ad+be+cf}{\norm{\xi_1}^2}, \qquad x_2 = \frac{\xi_2\cdot\xi_3}{\norm{\xi_2}^2}, \qquad x_3 = \frac{ag+bh+cj}{\norm{\xi_1}^2},
\end{align*}
\begin{align*}
	r = \norm{\xi_1}, \qquad y_1 = \frac{\norm{\xi_2}}{\norm{\xi_1}^2}, \qquad y_2 = \det A \frac{\norm{\xi_1}}{\norm{\xi_2}^2},
\end{align*}
and
\begin{align*}
	k = \Matrix{\frac{ce-bf}{\norm{\xi_2}}&\frac{af-cd}{\norm{\xi_2}}&\frac{bd-ae}{\norm{\xi_2}}\\[5pt] \frac{b(bd-ae)+c(cd-af)}{\norm{\xi_1}\norm{\xi_2}} & \frac{a(ae-bd)+c(ce-bf)}{\norm{\xi_1}\norm{\xi_2}} & \frac{a}{\norm{\xi_1}\norm{\xi_2}}\\[5pt] \frac{a(af-cd)+b(bf-ce)}{\norm{\xi_1}}&\frac{b}{\norm{\xi_1}}&\frac{c}{\norm{\xi_1}}} = \tilde{k}(s,t,u),
\end{align*}
with
\begin{align*}
	s = \frac{(bd-ae)\norm{\xi_1}+i(a(af-cd)+b(bf-ce))}{\sqrt{a^2+b^2}\norm{\xi_2}}, \qquad t = \frac{c+i\sqrt{a^2+b^2}}{\norm{\xi_1}}, \qquad u = \frac{-a+ib}{\sqrt{a^2+b^2}}.
\end{align*}

In case $a=b=0$, we may use
\begin{align*}
	s = \frac{e+id \sgn(c)}{\sqrt{d^2+e^2}}, \qquad t = \sgn(c), \qquad u = 1,
\end{align*}
though there are other choices.

For $g\in G$, we generally write $g=xyk$, taking $r=1$ as we may.

\section{The Whittaker functions}
\label{sect:WhittFuncs}
Let $I^d(\cdot, \mu)$ be the matrix-valued function on $G$ defined by
\begin{align}
\label{eq:IdDef}
	I^d(xyk, \mu) = p_{\rho+\mu}(y) \WigDMat{d}(k).
\end{align}
It can be shown that $I^d(\cdot, \mu)$ is an eigenfunction of the center of the Lie algebra of $G$.
The Jacquet-Whittaker function may be defined as an integral of this elementary function,
\begin{align}
\label{eq:LEWhittDef}
	W^d(g, \mu, \psi_n) =& \int_{U(\R)} I^d(w_l u g, \mu) \wbar{\psi_n(u)} du.
\end{align}
Using again the fact that $\abs{\WigD{d}{m'}{m}(k)} \le 1$, it is trivial to show absolute convergence on $\Re(\mu_1)>\Re(\mu_2)>\Re(\mu_3)$.
The previous definition may be considered a special case of the Jacquet-Whittaker function at a general Weyl element,
\begin{align}
\label{eq:WhittDef}
	W^d(g, w, \mu, \psi_n) =& \int_{\wbar{U}_w(\R)} I^d(w u g, \mu) \wbar{\psi_n(u)} du.
\end{align}
In addition to the computations of Bump \cite{Bump01} and others in the spherical case, the papers \cite{Miya01} and \cite{ManIshOda} give explicit forms of these functions in the minimal weight cases, which are the spherical, i.e. $d=0$, and $d=1$ cases.
(Miyazaki's results apply to the minimal $K$-types of the maximal parabolic Eisenstein series as well, but again, this will not assist in the analysis of the spectral expansion.)
The computations of this section generalize their results to all $d$.

We isolate the dependence on the components of the Iwasawa decomposition as follows:
First, by construction, we have
\begin{align}
\label{eq:WhittKFE}
	W^d(xyk, w, \mu, \psi_n) =& W^d(xy, w, \mu, \psi_n) \WigDMat{d}(k),
\end{align}
then a suitable substitution on $u$ gives
\begin{align}
\label{eq:WhittUFE}
	W^d(xy, w, \mu, \psi_n) =& \psi_{\what{n}}(x) W^d(y, w, \mu, \psi_n), \qquad \what{n} = \piecewise{
(0,0) & \If w=I, \\
(0,n_2) & \If w = w_2, w_4, \\
(n_1,0) & \If w = w_3, w_5, \\
n & \If w = w_l,}
\end{align}
and conjugating $uy \mapsto yu$ gives
\begin{align}
\label{eq:WhittYFE}
	W^d(y, w, \mu, \psi_n) =& p_{\rho+\mu^w}(y) W^d(I, w, \mu, \psi_{y n}), \\
\label{eq:WhittNFE}
	W^d(y, w, \mu, \psi_n) =& p_{-\rho-\mu^w}(\wtilde{n}) W^d(\wtilde{n} y, w, \mu, \psi_{\sgn(n)}),
\end{align}
where $\wtilde{n}$ is the $y$-matrix having coordinates $\wtilde{n}_i = \abs{n_i}$ if $n_i \ne 0$ and 1 otherwise.

The dependence on the signs of the character is given by
\begin{align}
\label{eq:WhittVFE}
	W^d(y, w, \mu, \psi_n^v) = \WigDMat{d}(wvw^{-1}) W^d(y, w, \mu, \psi_n) \WigDMat{d}(v),
\end{align}
or more generally,
\begin{align}
\label{eq:WhittVFE2}
	W^d(g, w, \mu, \psi_n^v) = \WigDMat{d}(wvw^{-1}) W^d(vg, w, \mu, \psi_n).
\end{align}
This can be seen by applying the Iwasawa decomposition to $(wuy)$ in
\[ wvuvy=(wvw^{-1}) (wuy) v. \]

\subsection{The degenerate Whittaker functions}
The cases of \eqref{eq:WhittDef} involving $w \ne w_l$ are called the degenerate Whittaker functions, and we may compute explicitly,
\begin{align}
\label{eq:DegenWhitt}
	& W^d(I, w, \mu, \psi_n) =\\
	& \piecewise{ I & \If w=I, \\[5pt]
		\WigDMat{d}(w_2 \vpmpm{+-}) \mathcal{W}^d(n_2, \mu_1-\mu_2) \WigDMat{d}(\vpmpm{+-}) & \If w=w_2, \\[5pt]
		\WigDMat{d}(w_5) \mathcal{W}^d(n_1, \mu_2-\mu_3) \WigDMat{d}(w_l) & \If w=w_3, \\[5pt]
		\WigDMat{d}(w_l \vpmpm{+-}) \mathcal{W}^d(0,\mu_2-\mu_3) \WigDMat{d}(w_3) \mathcal{W}^d(n_2, \mu_1-\mu_3) \WigDMat{d}(\vpmpm{+-}) & \If w=w_4, \\[5pt]
		\mathcal{W}^d(0,\mu_1-\mu_2) \WigDMat{d}(w_3) \mathcal{W}^d(n_1, \mu_1-\mu_3) \WigDMat{d}(w_l) & \If w=w_5.} \nonumber
\end{align}
These expressions continue to hold in case one or both $n_i=0$.

The proof at $w=w_4$, for example, starts by applying the Iwasawa decomposition \eqref{eq:FullIwasawa} with the trick \eqref{eq:w3Trick},
\begin{align*}
	& W^d(I, w_4, \mu, \psi_n) = \\
	& \WigDMat{d}(\vpmpm{+-} w_5) \int_{\R^2} \paren{\frac{\sqrt{1+u_2^2}}{1+u_2^2+u_3^2}}^{1-\mu_3} \paren{\frac{\sqrt{1+u_2^2+u_3^2}}{1+u_2^2}}^{1+\mu_1} \Dtildek{d}{\frac{u_3-i\sqrt{1+u_2^2}}{\sqrt{1+u_2^2+u_3^2}}} \\
	& \qquad \times \WigDMat{d}(w_3) \Dtildek{d}{\frac{-1+iu_2}{\sqrt{1+u_2^2}}} \e{-n_2 u_2} du_2 \, du_3.
\end{align*}
We have used that $\Dtildek{d}{i} = \WigDMat{d}(\vpmpm{+-} w_2)$.
Now substitute $u_3 \mapsto u_3 \sqrt{1+u_2^2}$, then
\begin{align*}
	W^d(I, w_4, \mu, \psi_n) =& \WigDMat{d}(\vpmpm{+-} w_5) \int_{\R} (1+u_3^2)^{\frac{-1+\mu_3-\mu_2}{2}} \Dtildek{d}{\frac{u_3-i}{\sqrt{1+u_3^2}}} du_3 \WigDMat{d}(w_3) \\
	& \qquad \times \int_{\R} (1+u_2^2)^{\frac{-1+\mu_3-\mu_1}{2}} \Dtildek{d}{\frac{1+iu_2}{\sqrt{1+u_2^2}}} \e{n_2 u_2} du_2 \Dtildek{d}{-1}.
\end{align*}
Applying the definition \eqref{eq:Wint} gives the desired evaluation.

\subsection{The Whittaker function at a degenerate character}
The following two special cases may also be evaluated explicitly:
\begin{align}
\label{eq:DegenCharWhitt}
	& W^d(I, w_l, \mu, \psi_n) = \\
	& \piecewise{
		\mathcal{W}^d(0,\mu_1-\mu_2) \WigDMat{d}(w_3) \mathcal{W}^d(0,\mu_1-\mu_3) \WigDMat{d}(w_5) \mathcal{W}^d(n_2,\mu_2-\mu_3) \WigDMat{d}(\vpmpm{+-}) & \If n_1=0, \\[5pt]
		\WigDMat{d}(w_l) \mathcal{W}^d(0,\mu_2-\mu_3) \WigDMat{d}(w_3) \mathcal{W}^d(0,\mu_1-\mu_3) \WigDMat{d}(w_5) \mathcal{W}^d(n_1,\mu_1-\mu_2) \WigDMat{d}(w_l) & \If n_2=0.} \nonumber
\end{align}
In case $n_1=n_2=0$, either expression is valid.

The proof in the case $n_1=0$ proceeds as before, but we must apply the sequence of substitutions
\begin{align}
\label{eq:CrazySub}
	u_1 \mapsto \frac{u_1}{\sqrt{1+u_2^2}}, \qquad u_3\mapsto u_3\sqrt{1+u_2^2}, \qquad u_1 \mapsto u_1 \sqrt{1+u_3^2}+u_2 u_3,
\end{align}
whose construction is explained in \cite[Section 5]{Me01}.
In case $n_2=0$, we conjugate by $w_l$ and apply the dual representation $\dualWigDMat{d}$, which reverses the rows and columns, and hence the roles of $u_1$ and $u_2$.
Then we apply the dual substitutions
\[ u_3 \mapsto u_1 u_2-u_3, \qquad u_2 \mapsto \frac{u_2}{\sqrt{1+u_1^2}}, \qquad u_3 \mapsto u_3 \sqrt{1+u_1^2}, \qquad u_2 \mapsto u_2 \sqrt{1+u_3^2}+u_1 u_3. \]

\subsection{Analytic continuation}
The analytic continuation in $\mu$ of the degenerate forms of the Whittaker function listed above follows from the known analytic continuation of the classical Whittaker function and the properties of the gamma function.
We wish to construct an integral representation of the non-degenerate long-element Whittaker function which clearly demonstrates its analytic continuation to a neighborhood of $\Re(\mu)=0$.
As above, we have
\begin{align}
\label{eq:LongEleWhittExpand}
	W^d(xyk, w, \mu, \psi_n) = \psi_n(x) p_{\rho+\mu^{w_l}}(y) W^d(I, w_l, \mu, \psi_{yn}) \WigDMat{d}(k),
\end{align}
and the dependence on the signs of $n$ is given by \eqref{eq:WhittVFE}, so it suffices to consider
\begin{align}
\label{eq:LongEleWhittExplicit}
	& W^d(I, w_l, \mu, \psi_y) =\\
	& \int_{U(\R)} \paren{1+u_2^2+u_3^2}^{\frac{-1+\mu_3-\mu_2}{2}} \paren{1+u_1^2+(u_3-u_1 u_2)^2}^{\frac{-1+\mu_2-\mu_1}{2}} \e{-y_1 u_1-y_2 u_2} \nonumber \\
	& \qquad \Dtildek{d}{\frac{-\sqrt{1+u_2^2+u_3^2}-i(u_1+u_2(u_1 u_2-u_3))}{\sqrt{1+u_2^2}\sqrt{1+u_1^2+(u_3-u_1 u_2)^2}}} \WigDMat{d}(w_3) \nonumber \\
	& \qquad \Dtildek{d}{\frac{-u_3-i\sqrt{1+u_2^2}}{\sqrt{1+u_2^2+u_3^2}}} \WigDMat{d}(w_3) \Dtildek{d}{\frac{1-iu_2}{\sqrt{1+u_2^2}}} du, \nonumber
\end{align}
with $y \in Y^+$.
For the moment, we assume that
\begin{align}
\label{eq:initWhittMBDef}
	\eta > \Re(\mu_1) > \Re(\mu_2) > \Re(\mu_3) > -\eta
\end{align}
for some small $\eta > 0$.
Applying the sequence of substitutions \eqref{eq:CrazySub}, this becomes
\begin{align}
\label{eq:ExplicitJacquet}
	& W^d(I, w_l, \mu, \psi_y) =\\
	& \int_{U(\R)} (1+u_1^2)^{\frac{-1+\mu_2-\mu_1}{2}} \paren{1+u_2^2}^{\frac{-1+\mu_3-\mu_2}{2}} \paren{1+u_3^2}^{\frac{-1+\mu_3-\mu_1}{2}} \Dtildek{d}{\frac{1+i u_1}{\sqrt{1+u_1^2}}} \WigDMat{d}(w_4) \nonumber \\
	& \qquad \Dtildek{d}{\frac{1-iu_3}{\sqrt{1+u_3^2}}} \WigDMat{d}(w_3) \Dtildek{d}{\frac{1-iu_2}{\sqrt{1+u_2^2}}} \e{-y_1 u_1 \frac{\sqrt{1+u_3^2}}{\sqrt{1+u_2^2}}-y_1 \frac{u_2 u_3}{\sqrt{1+u_2^2}} -y_2 u_2} du. \nonumber
\end{align}

For each of the three terms in the exponential, we apply the inverse Mellin transform
\begin{align}
\label{eq:PsiThetaInvMellin}
	\e{x} &= \lim_{\theta\to\frac{\pi}{2}^-} \frac{1}{2\pi i} \int_{\Re(t) = c} \abs{2\pi x}^{-t} e^{it\theta \sgn(x)} \Gamma\paren{t} \, dt,
\end{align}
for $x \ne 0$ and $c > 0$, which follows from the definition of the gamma function and Mellin inversion.
Collecting terms by the signs of each $u_i$ and applying the definition of the generalized beta function \eqref{eq:Bint}, we may write this as
\begin{align}
\label{eq:WhittMellin}
	& W^d(I, w_l, \mu, \psi_y) = p_{-\mu^{w_l}}(y) \int_{\Re(s)=(\frac{2}{10},\frac{2}{10})} (\pi y_1)^{-s_1} (\pi y_2)^{-s_2} \what{W}^d(s,\mu) \frac{ds}{(2\pi i)^2},
\end{align}
where
\begin{align}
\label{eq:WhittMellinEval}
	\what{W}^d(s,\mu) =&-2^{-s_1-s_2} (2\pi)^{\mu_1-\mu_3} \sum_{\delta\in\set{0,1}^3} (-i)^{\delta_1} i^{\delta_2+\delta_3} \sin\frac{\pi}{2}(\delta_2+s_2+\mu_3) \Gamma\paren{s_2+\mu_3} \\
	& \int_{\Re(t)=\frac{1}{10}} \sin\frac{\pi}{2}(\delta_1+t-\mu_1) \Gamma\paren{t-\mu_1} \sin\frac{\pi}{2}(\delta_3+ s_1-t) \Gamma\paren{s_1-t} \nonumber\\
	& \quad \mathcal{B}^d_{(-1)^{1-\delta_1}}\paren{1+\mu_1-t,t-\mu_2} \WigDMat{d}(w_4) \mathcal{B}^d_{(-1)^{\delta_3}}\paren{1-s_1+t, s_1-\mu_3} \nonumber\\
	& \quad \WigDMat{d}(w_3) \mathcal{B}^d_{(-1)^{\delta_2+\delta_3}}\paren{1-s_2-s_1-\mu_3+t, s_2-\mu_3-t} \frac{dt}{2\pi i}, \nonumber
\end{align}
after some rearranging.
We have dropped the limit in $\theta$, because the exponential decay factors in the $\mathcal{B}^d_\varepsilon$ functions give absolute convergence -- the first restricts the $t$ integral, the second the $s_1$ integral, and the third the $s_2$ integral.

Now we wish to extend the definition of  $\what{W}^d$ in $s$ and $\mu$, beyond the initial $\Re(s)=(\frac{2}{10},\frac{1}{10})$ and \eqref{eq:initWhittMBDef}.
Independent of the $t$ variable, $\what{W}^d$ has potential poles at $s_1=\mu_3-n$ and $s_2=-\mu_3-n$ for each $n \ge 0$.
For the $t$ integral itself, we may always deform the contour so that it passes to the left of the poles of the integrand at
\begin{align}
\label{eq:tContLeft}
	s_1+n, \qquad \mu_1+1+n, \qquad s_2-\mu_3+n,
\end{align}
and to the right of the poles at
\begin{align}
\label{eq:tContRight}
	\mu_1-n, \qquad \mu_2-n, \qquad s_1-1-n, \qquad s_1+s_2+\mu_3-1-n,
\end{align}
for each $n \ge 0$, unless a pole on the left of the contour intersects a pole on the right of the contour.
To examine the behavior in a small neighborhood of such a point, we shift the contour to the right past said point, picking up residues at the points \eqref{eq:tContLeft}.
The integral along the shifted contour is now analytic on the neighborhood, and it remains to investigate the behavior of the residues.

For the moment, we assume $\mu_i-\mu_j \notin \Z$, $i\ne j$, to avoid a discussion of the double poles that may occur.
By the explicit evaluation \eqref{eq:BexplicitEvalEven},\eqref{eq:BexplicitEvalOdd} of the $\mathcal{B}^d_\varepsilon$ function, it suffices to replace each $\mathcal{B}^d_{(-1)^\delta}(a,b)$ with a beta function of the form $B\paren{\frac{\delta+a}{2}+j, \frac{m-\delta+b}{2}-j}$, $\delta\in\set{0,1}$, $\abs{j} \le m \le d$.
In this manner, we see the residue at $t=s_1+n$, $n\equiv 1-\delta_3\pmod{2}$ (for $s_1+n$ in small neighborhoods of \eqref{eq:tContRight}), has at most simple poles at
\[ s_1=\mu_1-m, \qquad s_1=\mu_2-m, \]
the residue at $t=\mu_1+1+n$, $n\equiv 1-\delta_1\pmod{2}$ is zero, and the residue at $t=s_2-\mu_3+n$ has at most simple poles at 
\[ s_2=-\mu_1-m, \qquad s_2=-\mu_2-m, \]
$m \ge 0$.

Taking our analysis above, together with the bound \eqref{eq:BtildeBound}, implies
\begin{prop}
\label{prop:WhittMellin}
	The matrix-valued function $\what{W}^d(s,\mu)$ has meromorphic continuation to all $s\in\C^2$ and all $\mu$.
	$\what{W}^d_{m',m}(s,\mu)$ has at most simple poles at $s_1=\mu_i-n$ or $s_2=-\mu_i-n$, $n \ge 0$, provided $\mu_i-\mu_j \notin \Z$, $i\ne j$.
	For $d$, $\Re(s)$ and $\Re(\mu)$ in some fixed compact set, $\what{W}^d_{m',m}(s,\mu)$ has exponential decay in $\Im(s_i)$ for $\abs{\Im(s_i)} > 10 \max_j \abs{\Im(\mu_j)}$ and is otherwise polynomially bounded in the coordinates of $s$ and $\mu$.
\end{prop}
Note that we cannot, in general, remove the phrase ``at most'' without a significant analysis, because multiple terms in the sums \eqref{eq:BexplicitEvalEven},\eqref{eq:BexplicitEvalOdd} and the matrix multiplications in \eqref{eq:WhittMellinEval} may contribute to, and possibly cancel, a given residue.

Taking the contour $\Re(s)=(A_1,A_2)$ in \eqref{eq:WhittMellin}, we arrive at
\begin{cor}
\label{cor:WhitBd}
	There exists some $C > 0$ such that, for $d$ and $A_1, A_2 > \max_j \abs{\Re(\mu_j)}$ in some fixed compact set,
	\[ W^d_{m',m}(y, \mu,\psi_{1,1}) \ll y_1^{1-A_1} y_2^{1-A_2} \norm{\mu}^{C(1+A_1+A_2)}. \]
\end{cor}
Though the proposition only applies to $\mu_i-\mu_j\notin\Z$, we may extend by continuity -- this assumption was purely to avoid considering the double poles that will occur.
Of course, this result may be strengthened considerably, but the present work will only require very weak estimates here.

One might wonder about our use of the Mellin-Barnes integral representation of the Whittaker function.
In fact, by shifting polynomials among the gamma functions, the $t$ integral may be reduced to a finite sum of ``5-over-1'' forms as in Barnes' second lemma.
The sum of arguments of the gamma functions in the numerator minus that of the denominator will then be some integer.
In case the integer difference is non-positive, more shifting operations will produce a finite sum of integrals suitable for Barnes' second lemma -- this is sufficient to write $\what{W}^d_{m',m}(s,\mu)$ as a sum of quotients of gamma functions in a number of cases (e.g. $d=0,1$).
The trouble occurs when the integer difference is positive, and this does not seem to lend itself to the usual methods of evaluation.
Even were it possible to evaluate such integrals in general, the resulting expression for $\what{W}^d$ would likely be so complicated as to be unusable.

An alternate approach is to realize that the most interesting cases of the Whittaker function arise from applying raising operators to the $d=0,1$ functions.
The Mellin transform of these two functions may be evaluated explicitly, so that applying raising operators just produces polynomial multiples, but effective bounds for polynomials in four or five variables defined through multiple-term recurrence relations are quite difficult.

\subsection{Functional Equations}
Though they are not directly relevant to our proof of the spectral expansion, we give the functional equations (in $\mu$) of the matrix-valued Whittaker functions.
The functional equations of the various degenerate Whitttaker functions follow directly from the functional equation of the classical Whittaker function, so we just consider the long-element Whittaker function $W^d(y,\mu,\psi_{11})$.
It is sufficient to give the functional equations for the generators of the Weyl group action $\mu \mapsto \mu^{w_2}$ and $\mu \mapsto \mu^{w_3}$.
As we have the analytic continuation of all relevant integrals, it suffices to work formally.
To be precise, we may arrange that the equality of integrals at each step holds for $\mu$ in the region of absolute convergence of both sides; then getting between steps is an application of analytic continuation.

Starting from \eqref{eq:LongEleWhittExplicit}, we make the substitution
\[ u_1 \mapsto \frac{u_1 \sqrt{1+u_2^2+u_3^2}+u_2 u_3}{1+u_2^2}. \]
Then recognizing the classical integral \eqref{eq:Wint},
\begin{align}
\label{eq:WhittPartialEval}
	& W^d(I, \mu, \psi_y) =\\
	& \int_{\R^2} \paren{1+u_2^2+u_3^2}^{\frac{-1+\mu_3-\mu_1}{2}} \paren{1+u_2^2}^{\frac{-1+\mu_1-\mu_2}{2}} \mathcal{W}^d\paren{y_1 \frac{\sqrt{1+u_2^2+u_3^2}}{1+u_2^2},\mu_1-\mu_2} \Dtildek{d}{-1} \nonumber \\
	& \qquad \WigDMat{d}(w_3) \Dtildek{d}{\frac{-u_3-i\sqrt{1+u_2^2}}{\sqrt{1+u_2^2+u_3^2}}} \WigDMat{d}(w_3) \Dtildek{d}{\frac{1-iu_2}{\sqrt{1+u_2^2}}} \e{-y_1 \frac{u_2 u_3}{1+u_2^2}-y_2 u_2} du_2 \, du_3. \nonumber
\end{align}

Now apply the functional equation of the classical Whittaker function \eqref{eq:ClassWhittFE}, and applying \eqref{eq:WhittPartialEval} in reverse at $\mu^{w_2}$ gives
\begin{align}
\label{eq:FirstWhittFE}
	W^d(y, \mu, \psi_{11}) = \pi^{\mu_1-\mu_2} \Gamma^d_\mathcal{W}(\mu_2-\mu_1,+1) W^d(y, w_l, \mu^{w_2}, \psi_{11}),
\end{align}
after using \eqref{eq:LongEleWhittExpand}.

For $\mu \mapsto \mu^{w_3}$, we start with the dual representation, and apply now the substitutions $u_3 \mapsto u_1 u_2 - u_3$, then
\[ u_2 \mapsto \frac{u_2 \sqrt{1+u_1^2+u_3^2}+u_1 u_3}{1+u_1^2}, \]
arriving at the integral representation
\begin{align}
\label{eq:WhittDualPartialEval}
	& W^d(I, \mu, \psi_y) = \\
	& \int_{\R^2} \paren{1+u_1^2+u_3^2}^{\frac{-1+\mu_3-\mu_1}{2}} \paren{1+u_1^2}^{\frac{-1+\mu_2-\mu_3}{2}} \WigDMat{d}(w_l) \mathcal{W}^d\paren{y_2 \frac{\sqrt{1+u_1^2+u_3^2}}{1+u_1^2},\mu_2-\mu_3} \Dtildek{d}{-1} \nonumber \\
	& \qquad \WigDMat{d}(w_3) \Dtildek{d}{\frac{-u_3-i\sqrt{1+u_1^2}}{\sqrt{1+u_1^2+u_3^2}}} \WigDMat{d}(w_3) \Dtildek{d}{\frac{1+iu_1}{\sqrt{1+u_1^2}}} \WigDMat{d}(w_l) \e{-y_1 u_1-y_2 \frac{u_1 u_3}{1+u_1^2}} du_1 \, du_3. \nonumber
\end{align}

From this follows the functional equation
\begin{align}
\label{eq:SecondWhittFE}
	W^d(y, \mu, \psi_{11}) =& \pi^{\mu_2-\mu_3} \WigDMat{d}(w_l) \Gamma^d_\mathcal{W}(\mu_3-\mu_2,+1) \WigDMat{d}(w_l) W^d(y, \mu^{w_3}, \psi_{11}).
\end{align}

Now set
\begin{align*}
	T^d(w_2,\mu) =&\pi^{\mu_1-\mu_2} \Gamma^d_\mathcal{W}(\mu_2-\mu_1,+1), &T^d(w_3,\mu) =& \pi^{\mu_2-\mu_3} \WigDMat{d}(w_l) \Gamma^d_\mathcal{W}(\mu_3-\mu_2,+1) \WigDMat{d}(w_l),
\end{align*}
\begin{align*}
	T^d(I,\mu) =&I, &T^d(w_4,\mu) =& T^d(w_3,\mu)T^d(w_2,\mu^{w_3}), &T^d(w_5,\mu)=& T^d(w_2,\mu)T^d(w_3,\mu^{w_2}),
\end{align*}
and
\[ T^d(w_l,\mu)=T^d(w_2,\mu)T^d(w_4,\mu^{w_2}). \]
It can be shown, by relating $T^d(w,\mu)$ to $W^d(I,w,\mu,\psi_{00})$ and applying \eqref{eq:DegenCharWhitt}, that also
\[ T^d(w_l,\mu)=T^d(w_3,\mu)T^d(w_5,\mu^{w_3}). \]
(We will examine one case of this in more detail in section \ref{sect:minimalFEs}.)
Then the functional equations take the form
\begin{prop}
\label{prop:WhittFEs}
	For each $w\in W$,
	\[ W^d(g, \mu, \psi_n) = T^d(w,\mu) W^d(g, \mu^w, \psi_n). \]
\end{prop}

It is somewhat more useful to consider the functional equations of $\Sigma^d_\chi W^d(y,\mu,\psi_{11})$, which functions transform according to the character $\chi$ on $V$,
\[ \Sigma^d_{V,\chi} W^d(y,\mu,\psi_n^v) = \chi(w_l v w_l) \Sigma^d_{V,\chi} W^d(y,\mu,\psi_n) \WigDMat{d}(v), \]
and these are the Whittaker functions which occur in the Fourier expansion of the Maass forms, assuming we have enforced a similar condition on the Maass forms themselves.
We will not work out these functional equations here, except to note that those for the trivial character $\chi=1$ follow from the functional equations of the minimal parabolic Eisenstein series, see Corollary \ref{cor:EvenWhittFEs}.
We leave the remaining equations to the interested reader.

\subsection{Asymptotics}
\label{sect:WhittAsymps}
On the region of absolute convergence, say
\[ \Re(\mu_1)-\epsilon > \Re(\mu_2) > \Re(\mu_3)+\epsilon, \]
it is easy to see from the Jacquet integral \eqref{eq:LEWhittDef} and Proposition \ref{prop:WhittMellin} that the first term asymptotic as $y \to 0$ is
\begin{align}
\label{eq:WhittFirstAsymp}
	W^d(y,\mu,\psi_{11}) \sim W^d(y,\mu,\psi_{00}),
\end{align}
in the sense that the difference of the two sides is a matrix whose components are
\[ \ll_{d,\mu} p_{\rho+\Re(\mu^{w_l})}(y) (y_1^\epsilon+y_2^\epsilon). \]

This extends by analytic continuation and Proposition \ref{prop:WhittFEs} to
\begin{align}
\label{eq:WhittSecondAsymp}
	W^d(y,\mu,\psi_{11}) \sim \sum_{w\in W} T^d(w,\mu) W^d(y,\mu^w,\psi_{00})
\end{align}
on $\Re(\mu) = 0$ with $\mu_i \ne \mu_j$ for $i\ne j$, and by Proposition \ref{prop:WhittMellin}, we may strengthen the error bound to
\[ \ll_{d,\mu} y_1 y_2 (y_1+y_2). \]

\subsection{Fourier expansions}
\label{sect:FourierExp}
Smooth automorphic functions on $GL(3)$ have a Fourier expansion due to Piatetski-Shapiro \cite{PS01} and Shalika \cite{Shal01}, see the proof in \cite[Theorem 5.3.2]{Gold01}.
(Note that right $K$ invariance is not used there, and the statement for non-cuspidal functions is false; compare with \eqref{eq:FourierExpansion}.)
This readily applies to smooth vector- and matrix-valued functions on $G$ which are left-invariant by $\Gamma$.
If $\phi$ is such a function, we define its Fourier coefficients as
\[ \wtilde{\phi}_n(g) = \int_{U(\Z)\backslash U(\R)} \phi(ug) \wbar{\psi_n(u)} du, \]
and the Fourier expansion takes the form
\begin{align}
\label{eq:FourierExpansion}
	\phi(g) =& \sum_{n_1\in\Z} \wtilde{\phi}_{(n_1,0)}(g)+\sum_{\gamma\in U(\Z)\backslash SL(2,\Z)} \sum_{n \in \Z \times \N} \wtilde{\phi}_n(\gamma g).
\end{align}
To avoid excessive notation in subscripts and function arguments, we have embedded $SL(2,\Z)$ into $SL(3,\Z)$ in the upper left corner, $\SmallMatrix{*&*\\ *&*\\&&1}$, and $U(\Z)\backslash SL(2,\Z)$ is shorthand for $(U(\Z)\cap SL(2,\Z))\backslash SL(2,\Z)$.

A matrix-valued automorphic function $\phi(xyk) = \phi(xy)\WigDMat{d}(k)$ which is an eigenfunction of the $SL(3,\R)$ differential operators will have Fourier coefficients which are multiples of the Whittaker functions above.
Precisely, if $\phi$ shares the eigenvalues of $p_{\rho+\mu_\phi}$, then
\begin{align}
\label{eq:FWCoefsDef}
	\wtilde{\phi}_n(g) =& \sum_{w\in W} \delta_{n,w} \frac{\rho_\phi(w, n)}{\wtilde{n}_1 \wtilde{n}_2} W^d(\wtilde{n} g, w, \mu_\phi, \psi_{\sgn(n)}),
\end{align}
with $\wtilde{n}$ again being the $y$-matrix having coordinates $\wtilde{n}_i = \abs{n_i}$ if $n_i \ne 0$ and 1 otherwise, and $\delta_{n,w}$ is one if $\psi_n$ is trivial on $U_w(\R)$, i.e. when one of
\[ \begin{array}{lclclcl}
	w=I, n=0, &\text{or}& w=w_2, n_1=0, &\text{or}& w=w_3, n_2=0 &\text{or}& \\
	w=w_4, n_1=0, &\text{or}& w=w_5, n_2=0, &\text{or}& w=w_l, 
\end{array} \]
is true, and zero otherwise.
The matrices $\rho_\phi(w,n)$ are called the Fourier-Whittaker coefficients of $\phi$.
We will not seek to prove this here, but rather demonstrate it for the Eisenstein series by computing their Fourier coefficients directly.

The absolute convergence of the $n$ and $\gamma$ sums in the Fourier expansion follows from applying the usual rapid convergence of the classical Fourier expansion of a smooth function during the proof, but we require a direct proof of this fact.
Essentially, we will give the trivial bound on the supremum norm coming from the Fourier expansion.

We split \eqref{eq:FourierExpansion} into pieces by the degeneracy of the characters,
\[ \phi(g) = \wtilde{\phi}_{00}(g)+\phi_1(g)+\phi_2(g)+\phi_3(g), \]
where
\begin{align*}
	\phi_1(g) =& \sum_{n_1\ne 0} \wtilde{\phi}_{(n_1,0)}(g), \\
	\phi_2(g) =& \sum_{\gamma\in U(\Z)\backslash SL(2,\Z)} \sum_{n_2 \in \N} \wtilde{\phi}_{(0,n_2)}(\gamma g), \\
	\phi_3(g) =& \sum_{\gamma\in U(\Z)\backslash SL(2,\Z)} \sum_{\substack{n \in \Z \times \N\\ n_1 \ne 0}} \wtilde{\phi}_n(\gamma g).
\end{align*}

\begin{prop}
\label{eq:FourierExpBound}
	Suppose the entries of $\rho_\phi(w,n)$ are bounded by $(\wtilde{n}_1 \wtilde{n}_2)^A$, for some $A>0$ and let $\mu=\mu_\phi$, $c = \max_i \abs{\Re(\mu_i)}$, then the $n$ and $\gamma$ sums of \eqref{eq:FourierExpansion} converge absolutely, and for any $B > \Max{c,A}+1$,
	\begin{align*}
		\wtilde{\phi}_{00}(xyk) \ll_{d, \mu}& (y_1+y_1^{-1})^{1+c} (y_2+y_2^{-1})^{1+c}, \\
		\phi_1(xyk) \ll_{d, \mu}& (y_2+y_2^{-1})^{1+c} y_1^{-B}, \\
		\phi_2(xyk) \ll_{d, \mu}& (y_1+y_1^{-1})^{1+c} y_2^{-B}, \\
		\phi_3(xyk) \ll_{d, \mu}&\, y_1^{\frac{3}{2}+\epsilon} y_2^\epsilon (y_2+y_2^{-1}) (y_1^2 y_2)^{-B}.
	\end{align*}
	The dependence of the implied constants on $\Im(\mu)$ is polynomial.
\end{prop}
\begin{proof}
For $\gamma=\SmallMatrix{*&*\\a&b}$, the $y$ component of the Iwasawa decomposition of $\gamma xyk=x^* y^* k^*$ has coordinates
\begin{align*}
	y_1^* =& y_1 \sqrt{y_2^2 a^2+(a x_2+b)^2}, \\
	y_2^* =& \frac{y_2}{y_2^2 a^2+(a x_2+b)^2}.
\end{align*}

Entries of $\phi_3(xyk)$ are bounded by sums of at most $2d+1$ terms of the form
\[ \sum_{\gamma\in U(\Z)\backslash SL(2,\Z)} \sum_{\substack{n \in \Z \times \N\\ n_1 \ne 0}} \abs{n_1 n_2}^{A-1} \abs{W^d_{m',m}(\wtilde{n} y^* k^*, w_l, \mu, \psi_{\sgn(n)})}. \]
Taking $A_1=\frac{1}{2}+\epsilon+2B$, $A_2=B$, $B>\Max{c,A}+1$ in corollary \ref{cor:WhitBd} and applying
\[ (1+y_2^{-1})^2 (y_2^2 a^2+(a x_2+b)^2) \ge a^2+(a x_2+b)^2, \]
the previous display is bounded by
\[ y_1^{\frac{1}{2}-\epsilon-2B} y_2^{1-B} (1+y_2^{-1})^{2+\epsilon} \sum_{\substack{a,b\in\Z\\(a,b)=1}} (a^2+(a x_2+b)^2)^{-1-\epsilon} \sum_{\substack{n \in \Z \times \N\\ n_1 \ne 0}} \abs{n_1}^{\frac{1}{2}+\epsilon} \frac{\abs{n_1 n_2}^A}{(n_1^2 n_2)^B}. \]
The $a,b$ sum converges absolutely, and we will see this in more detail later.

The Whittaker functions in $\wtilde{\phi}_{00}(g)$ are just power functions, so the bound there is obvious.

For the intermediate terms, we must deal with the classical Whittaker function.
The function $W_{\alpha,\beta}(y)$ is known to be entire in $\alpha,\beta$, and this implies the same for the function $\mathcal{W}_m(y,u)$.
There are no simple, uniform bounds in the usual references, but we note that applying the definition \eqref{eq:WintOld} for $\Re(u) > \epsilon > 0$, the functional equation \eqref{eq:ClassWhittFE} for $\Re(u) < -\epsilon$ and Phragm\'en-Lindel\"of in the middle yields
\[ \mathcal{W}_m(y,u) \ll \paren{\frac{\abs{y}}{1+\abs{u}}}^{\Min{0,\Re(u)}+o(1)}, \]
for $m$ and $\abs{\Re(u)}$ in some fixed compact set.
Note that we are always away from the poles of $\Gamma^d_{\mathcal{W},m,m}$, so these do not interfere with our analysis.
Multiple integration by parts similarly yields $\mathcal{W}_m(y,u) \ll_{m,u} \abs{y}^{-B}$ for any $B>\Max{0,-\Re(u)}$.
Here we are making liberal use of Stirling's formula.

Armed with these bounds on the $\mathcal{W}_m(y,u)$ function, the proofs for $\phi_1$ and $\phi_2$ are the same as for $\phi_3$.

\end{proof}

\section{The minimal parabolic Eisenstein series}
\label{sect:MinPara}
We define the minimal parabolic Eisenstein series attached to $\WigDMat{d}$ by
\begin{align}
\label{eq:MinParaEisenDef}
	E^d(g, \mu) = \sum_{\gamma\in U(\Z)\backslash \Gamma} I^d(\gamma g, \mu),
\end{align}
using the matrix-valued function $I^d(\cdot,\mu)$ defined by \eqref{eq:IdDef}.
Since $\WigDMat{d}(k)$ lies in $SO(2d+1,\C)$, its entries satisfy $\abs{\WigD{d}{m'}{m}(k)} \le 1$, and so the absolute convergence of this series on $\Re(\mu_1-\mu_2),\Re(\mu_2-\mu_3)>1$ follows from the analysis of Bump \cite[Ch. 7]{Bump01}, but we will show this directly.

\subsection{Absolute convergence}
We need a weak bound on the $y$ dependence of the Eisenstein series in the region of absolute convergence.
This follows from the fact that
\begin{align*}
	\abs{E^d_{m',m}(xyk,\mu)} \le& E(xy, \Re(\mu)),
\end{align*}
where $E(g,\mu) = E^0_{0,0}(g,\mu)$ is the spherical Eisenstein series.
Now using the Pl\"ucker coordinates in the Iwasawa decomposition, we have
\begin{align*}
	E(xy, \mu) =& p_{\rho+\mu}(y) \sum_{\substack{A, B, C\in \Z^2\\(A_i,B_i,C_i)=1\\A_1 C_2+B_1 B_2+C_1 A_2=0}} \xi_1^{\frac{-1+\mu_3-\mu_2}{2}} \xi_2^{\frac{-1+\mu_2-\mu_1}{2}},
\end{align*}
where
\begin{align*}
	\xi_1 =& y_1^2 y_2^2 A_1^2+y_1^2 (A_1 x_2+B_1)^2+(A_1 x_3+B_1 x_1+C_1)^2, \\
	\xi_2 =& y_1^2 y_2^2 A_2^2+y_2^2 (A_2 x_1-B_2)^2+(A_2 (x_1 x_2-x_3)-B_2 x_2+C_2)^2.
\end{align*}
We separate the $y$ dependence by the inequalities
\begin{align*}
	(1+y_1^{-1})^2 (1+y_2^{-1})^2 \xi_1 \ge& A_1^2+(A_1 x_2+B_1)^2+(A_1 x_3+B_1 x_1+C_1)^2, \\
	(1+y_1^{-1})^2 (1+y_2^{-1})^2 \xi_2 \ge& A_2^2+(A_2 x_1-B_2)^2+(A_2 (x_1 x_2-x_3)-B_2 x_2+C_2)^2.
\end{align*}

To prove absolute convergence, we may drop the condition $A_1 C_2+B_1 B_2+C_1 A_2=0$, and replace $(A_i, B_i, C_i)=1$ with $A_i B_i C_i \ne 0$, so it suffices to prove convergence of
\[ \sum_{\substack{a \in \Z^3\\a_1 a_2 a_3 \ne 0}} (a_1^2+(a_1 x_2+a_2)^2+(a_1 x_3+a_2 x_1+a_3)^2)^{-1-c}, \qquad c > 0. \]
For $a_1 \ne 0$, we may compare with the integral
\[ \int_1^\infty \int_{-\infty}^\infty \int_{-\infty}^\infty (a_1^2+(a_1 x_2+a_2)^2+(a_1 x_3+a_2 x_1+a_3)^2)^{-1-c} da_3 \, da_2 \, da_1 < \infty, \]
(more precisely: for fixed $a_1$ and $a_2$, the $a_3$ summand changes from increasing to decreasing at exactly one point, which may be analyzed separately, and similarly for two and three points in the $a_2$ and $a_1$ sums, respectively) and the case $a_1 = 0$ follows similarly by splitting on $a_2=0$.

Thus
\begin{align*}
	E(xy, \mu) \ll_\eta& p_{\rho+\Re(\mu)}(y) (1+y_1^{-1})^{2+\Re(\mu_1-\mu_3)} (1+y_2^{-1})^{2+\Re(\mu_1-\mu_3)},
\end{align*}
and in general,
\begin{align}
\label{eq:TrivMinEisenBd}
	E^d_{m',m}(xyk,\mu) \ll_\eta y_1^{1-\Re(\mu_3)} y_2^{1+\Re(\mu_1)} (1+y_1^{-1})^{2+\Re(\mu_1-\mu_3)} (1+y_2^{-1})^{2+\Re(\mu_1-\mu_3)},
\end{align}
on the region of absolute convergence, $\Re(\mu_1-\mu_2),\Re(\mu_2-\mu_3)>1+\eta$, $\eta>0$.

\subsection{Fourier coefficients}
\label{sect:MinEisenFour}
Applying the Bruhat decomposition as in \cite[Prop. 10.3.6]{Gold01}, the Fourier coefficients of the Eisenstein series are given by
\begin{align*}
	\rho_{E^d}(n,g,\mu) :=& \int_{U(\Z)\backslash U(\R)} E^d(ug,\mu) \wbar{\psi_n(u)} du \\
	=& \sum_{w\in W} \delta_{n,w} \sum_{v\in V} \sum_{c_1,c_2\in\N} p_{\rho+\mu}(c) \sum_{bcwb'\in U(\Z)\backslash \Gamma/V\wbar{U}_w(\Z)} \psi_n^v(b') \int_{\wbar{U}_w(\R)} I^d(wuvg, \mu) \wbar{\psi_n^v(u)} du \\
	=& \sum_{w\in W} \delta_{n,w} \sum_{v\in V} \zeta(w, \mu, \psi_n^v) W^d(vg,w,\mu,\psi_n^v),
\end{align*}
where
\[ \zeta(w, \mu, \psi_n^v) := \sum_{c_1,c_2\in\N} p_{\rho+\mu}(c) \sum_{bcwb'\in U(\Z)\backslash \Gamma/V\wbar{U}_w(\Z)} \psi_n^v(b'). \]

Of course, the Fourier coefficients satisfy
\[ \rho_{E^d}(n,xyk,\mu) = \psi_n(x) \rho_{E^d}(n,y,\mu) \WigDMat{d}(k), \]
so it suffices to examine
\begin{align*}
	\rho_{E^d}(n,y,\mu) =& \sum_{w\in W} \delta_{n,w} \sum_{v\in V} \zeta(w, \mu, \psi_n^v) W^d(vy,w,\mu,\psi_n^v) \\
	=& \sum_{w\in W} \delta_{n,w} \sum_{v\in V} \zeta(w, \mu, \psi_n^v) \WigDMat{d}(wvw^{-1}) W^d(y, w, \mu, \psi_n).
\end{align*}

Bump \cite[Ch 7]{Bump01} has computed the function $\zeta(w, \mu, \psi_n^v)$ very explicitly.
Define
\[ \zeta_E(\mu) = \prod_{j<k} \zeta(1+\mu_j-\mu_k), \]
\[ \Gamma_\R(s) = \pi^{-s/2} \Gamma\paren{\tfrac{s}{2}}, \]
and
\[ \sigma_s(n) = \sum_{a|n} a^s. \]
Let $\sigma_{s_1, s_2}(n_1, n_2)$ be the multiplicative function defined by
\begin{align}
\label{eq:MinHecke}
	\sigma_{s_1, s_2}(p^\alpha, p^\beta) = p^{-s_2\alpha} S_{\alpha,\beta}\paren{p^{s_1+s_2},p^{s_2}},
\end{align}
for prime numbers $p$, using \eqref{eq:SchurDef}.
Finally, let
\begin{align*}
	\Gamma_E(w,\mu) =& \piecewise{
1 & \If w=I, \\
\frac{\Gamma_\R\paren{1-\mu_1+\mu_2}}{\Gamma_\R\paren{\mu_1-\mu_2}} & \If w=w_2, \\[5pt]
\frac{\Gamma_\R\paren{1-\mu_2+\mu_3}}{\Gamma_\R\paren{\mu_2-\mu_3}} & \If w=w_3, \\[5pt]
\frac{\Gamma_\R\paren{1-\mu_2+\mu_3} \Gamma_\R\paren{1-\mu_1+\mu_3}}{\Gamma_\R\paren{\mu_2-\mu_3} \Gamma_\R\paren{\mu_1-\mu_3}} & \If w=w_4, \\[5pt]
\frac{\Gamma_\R\paren{1-\mu_1+\mu_2} \Gamma_\R\paren{1-\mu_1+\mu_3}}{\Gamma_\R\paren{\mu_1-\mu_2} \Gamma_\R\paren{\mu_1-\mu_3}} & \If w=w_5, \\[5pt]
\frac{\Gamma_\R\paren{1-\mu_1+\mu_2} \Gamma_\R\paren{1-\mu_2+\mu_3} \Gamma_\R\paren{1-\mu_1+\mu_3}}{\Gamma_\R\paren{\mu_1-\mu_2} \Gamma_\R\paren{\mu_2-\mu_3} \Gamma_\R\paren{\mu_1-\mu_3}} & \If w=w_l.}
\end{align*}

Then for $n_1 n_2 \ne 0$,
\begin{align}
\label{eq:MinEisenNDG}
	\zeta_E(\mu) \zeta(w_l, \mu, \psi_n^v) = \sigma_{\mu_2-\mu_1,\mu_3-\mu_2}(\abs{n_1}, \abs{n_2}),
\end{align}
for $n_1=n_2=0$ and $w\in W$,
\begin{align}
\label{eq:MinEisenDG1}
	\zeta_E(\mu) \zeta(w, \mu, \psi_n^v) = \Gamma_E(w,\mu) \zeta_E(\mu^w),
\end{align}
for $n_1=0, n_2\ne 0$ and $w\in\set{w_2,w_4,w_l}$,
\begin{align}
\label{eq:MinEisenDG2}
	\zeta_E(\mu) \zeta(w, \mu, \psi_n^v) = \Gamma_E(w w_2,\mu) \zeta(1+\mu^w_2-\mu^w_3)\zeta(1+\mu^w_1-\mu^w_3)\sigma_{\mu^w_1-\mu^w_2}(\abs{n_2}),
\end{align}
and for $n_1\ne 0, n_2=0$ and $w\in\set{w_3,w_5,w_l}$,
\begin{align}
\label{eq:MinEisenDG3}
	\zeta_E(\mu) \zeta(w, \mu, \psi_n^v) = \Gamma_E(w w_3,\mu) \zeta(1+\mu^w_1-\mu^w_2) \zeta(1+\mu^w_1-\mu^w_3) \sigma_{\mu^w_2-\mu^w_3}(\abs{n_1}).
\end{align}
This is a somewhat complicated change of notation from \cite{Bump01}:
First, we use the Langlands parameters $\mu$ in place of $\nu$; the conversion is $3\nu=(1+\mu_2-\mu_3,1+\mu_1-\mu_2)$.
Second, we have applied the functional equation to the zeta factors as necessary, so the dependence there may be expressed in terms of the Weyl action; this results in the gamma factors listed above (the poles of the gamma factors are then cancelled by the trivial zeros of the zeta function).
Third, by leaving $\psi_n$ in the Whittaker function, we have absorbed a factor $p_{\mu^w}(\wtilde{n})$ into the Fourier coefficients $\zeta(w, \mu, \psi_n^v)$.
It is simplest to compare $\zeta(w, \mu, \psi_n^v)$ to the equation at the bottom of his page 111, in the long-element case; note that he has absorbed the $\zeta(3\nu_1) \zeta(3\nu_2)$ by dropping $(A_1, B_1, C_1)=(A_2,B_2,C_2)=1$.

Notice that $\zeta(w, \mu, \psi_n^v)$ is independent of $v$, so
\begin{align}
\label{eq:EisenFourier}
	\rho_{E^d}(n,y,\mu) =& \abs{V} \sum_{w\in W} \delta_{n,w} \zeta(w, \mu, \psi_n) \Sigmachi{d}{++} W^d(y, w, \mu, \psi_n).
\end{align}

The Fourier-Whittaker coefficients in the normalization of \eqref{eq:FWCoefsDef} are then
\[ \rho_{E^d}(w, n,\mu) = \abs{V} \zeta(w, \mu, \psi_n) p_{-\mu^w}(\wtilde{n}) \Sigmachi{d}{++}. \]

\subsection{Initial continuation of the Eisenstein series}
Combining the explicit evaluation of the Fourier-Whittaker coefficients in the previous section with the results of section \ref{sect:FourierExp}, we see that the Eisenstein series has analytic continuation to all $\mu$, unless the Fourier coefficients themselves have poles.
Having never seen the definition \eqref{eq:SchurDef} of the multiplicative function $\sigma_{s_1,s_2}(p^\alpha, p^\beta)$, one might worry about poles there, but as a function of $p^{s_1}, p^{s_2}$, for fixed $\alpha, \beta$, these are polynomials, and hence entire.

Thus we are left to consider \eqref{eq:MinEisenDG1}-\eqref{eq:MinEisenDG3}.
Except at $\mu_i-\mu_j=1, i<j$, the poles of $\Gamma_E(w,\mu)$ are always cancelled by the zeta functions on the right-hand sides, but the non-trivial zeros of the zeta functions on the left-hand sides force us to restrict to a zero-free region:
If $1>c>0$ is a constant such that $\zeta(1+s) \ne 0$ on
\begin{align}
\label{eq:FormalZeroFree}
	\Re(s) > -\frac{c}{\log\paren{1+\abs{\Im(s)}}},
\end{align}
(c.f. \cite[Theorem 5.10]{IK})
then the Fourier expansion converges to a holomorphic function of $\mu$ on
\[ \Re(\mu_i-\mu_j) > -\frac{c}{\log\paren{1+\abs{\Im(\mu_i-\mu_j)}}}, \qquad i < j, \]
except for possible simple poles at $\mu_i-\mu_j=1, i<j$.

We will discuss the poles in greater detail in section \ref{sect:ResSpec}.

\subsection{Functional equations}
\label{sect:minimalFEs}
\subsubsection{The functional equations of the constant terms}
Define
\[ M^d(w, \mu) = \Gamma_E(w,\mu) \frac{\zeta_E(\mu^w)}{\zeta_E(\mu)} \Sigmachi{d}{++} W^d(I, w, \mu, \psi_{00}), \]
so that the $1,1,1$ constant term of the minimal parabolic Eisenstein series is
\[ \rho_{E^d}(0,y,\mu) = \abs{V} \sum_{w\in W} p_{\rho+\mu^w}(y) M^d(w,\mu). \]

The $1,1,1$ constant term satisfies the functional equations
\begin{align}
\label{eq:111FEs}
	M^d(w', \mu) \rho_{E^d}(0,y,\mu^{w'}) =& \rho_{E^d}(0,y,\mu)
\end{align}
on the intersection of the zero-free regions of the relevant zeta functions:
\begin{align}
\label{eq:FEset}
	\mathcal{E} = \set{\mu\setdiv \abs{\Re(\mu_j-\mu_k)} < \frac{c}{\log\paren{1+\abs{\Im(\mu_j-\mu_k)}}}}.
\end{align}
We now proceed to verify these functional equations; in terms of the $M^d(w,\mu)$ matrix, they are equivalent to
\begin{align}
\label{eq:MmatFEs}
	M^d(w', \mu) M^d(w,\mu^{w'}) = M^d(w'w,\mu).
\end{align}
The sign dependence of the Whittaker function \eqref{eq:WhittVFE} tells us that
\[ W^d(I, w, \mu, \psi_{00}) \Sigmachi{d}{++} = \Sigmachi{d}{++} W^d(I, w, \mu, \psi_{00}), \]
and the quotient of zeta functions is trivial to deal with, so we may reduce to checking
\[ \Gamma_E(w',\mu) \Gamma_E(w,\mu^{w'}) \Sigmachi{d}{++} W^d(I, w', \mu, \psi_{00}) W^d(I, w, \mu^{w'}, \psi_{00}) = \Gamma_E(w'w,\mu) \Sigmachi{d}{++} W^d(I, w'w, \mu, \psi_{00}). \]
Furthermore, using the known generators and relations for the symmetric group on three letters, it suffices to check this for the following triples $(w',w,w'w)$:
\[\begin{array}{cccccc}
	(w_2, w_2, I), & (w_3, w_3, I), & (w_2, w_3, w_5), & (w_3, w_2, w_4), & (w_2, w_4, w_l), & (w_3, w_5, w_l).
\end{array} \]
By \eqref{eq:WDtildek} and \eqref{eq:W0v}, we know that $\mathcal{W}^d(0, u)$ commutes with $\WigDMat{d}(w_2)$ and $\WigDMat{d}(v_{\pm\pm})$; the presence of $\Sigmachi{d}{++}$ means we may ignore the latter entirely.
Lastly, by the explicit description of $\Sigmachi{d}{++}$ in \eqref{eq:Sigmadplusplus}, we need only consider the even rows (and columns) of the above matrix relation.

For the $(w_2,w_2,I)$ case, after writing out the explicit form of $W^d$ in \eqref{eq:DegenWhitt}, we may freely conjugate the factors $\WigDMat{d}(w_2)$ to the left, whereupon they cancel.
Then it is trivial to verify that the even entries of the diagonal matrix
\begin{align}
\label{eq:mathcalWFE}
	\frac{\Gamma_\R\paren{1-\mu_1+\mu_2}}{\Gamma_\R\paren{\mu_1-\mu_2}} \frac{\Gamma_\R\paren{1-\mu_2+\mu_1}}{\Gamma_\R\paren{\mu_2-\mu_1}} \mathcal{W}^d(0, \mu_1-\mu_2) \mathcal{W}^d(0, \mu_2-\mu_1)
\end{align}
are all equal to one.
The case $(w_3, w_3, I)$ is similarly trivial, after realizing the commutativity relations
\[ \Sigmachi{d}{++} \WigDMat{d}(w) = \WigDMat{d}(w) \Sigmachi{d}{++}, \]
for any Weyl group element $w$.
For the remaining cases, the two statements
\[ \Gamma_E(w',\mu) \Gamma_E(w,\mu^{w'}) = \Gamma_E(w'w,\mu), \]
and
\[ \Sigmachi{d}{++} W^d(I, w', \mu, \psi_{00}) W^d(I, w, \mu^{w'}, \psi_{00}) = \Sigmachi{d}{++} W^d(I, w'w, \mu, \psi_{00}) \]
hold independently; this is trivial to verify applying the first form of \eqref{eq:DegenCharWhitt} to $(w_2, w_4, w_l)$ and the second to $(w_3, w_5, w_l)$.

Of course Langlands \cite{Langlands02} gave a much more general and abstract argument, but given the current tools, we find the direct verification to be quite pleasant.

It is possible to show the functional equations using only the $n=0$ constant term from just the two transpositions $\mu \mapsto \mu^{w_2}$ and $\mu \mapsto \mu^{w_3}$ (which then compose to give the full group), but we would like to show the more general form $\mu \mapsto \mu^w$ directly by applying the functional equations of the partially degenerate terms $n_1=0\ne n_2$ and $n_1\ne0=n_2$ as well.
These functional equations should take a term of the form
\[ M(w',\mu) \zeta(w, \mu^{w'}, \psi_n) \Sigmachi{d}{++} W^d(y, w, \mu^{w'}, \psi_n) \]
to
\[ \zeta(w'', \mu, \psi_n) \Sigmachi{d}{++} W^d(y, w'', \mu, \psi_n), \]
for some $w'' \in W$, but the $w''$ is now determined by multiplication in a quotient group, since the delta function $\delta_{n,w}$ in the Fourier-Whittaker expansion restricts the terms involved.
We may drop the occurances of the Riemann zeta function from $\zeta(w'', \mu, \psi_n)$ since it will turn out, say for $n_1=0$ that $w''\in\set{w'wI,w'ww_2}$ (so $w'' \equiv w'w$ modulo $w_2$ on the right), and we may write
\[ \zeta(1+\mu^w_2-\mu^w_3)\zeta(1+\mu^w_1-\mu^w_3) = \zeta(1+\mu^{wI}_1-\mu^{wI}_3) \zeta(1+\mu^{ww_2}_1-\mu^{ww_2}_3). \]
Also, we know (again for $n_1=0$)
\begin{align*}
	&\Gamma_E(w',\mu) \Gamma_E(ww_2, \mu^{w'}) \Sigmachi{d}{++} W^d(I, w', \mu, \psi_{00}) W^d(I, w w_2, \mu^{w'}, \psi_{00}) \\
	&= \Gamma_E(w'ww_2, \mu) \Sigmachi{d}{++} W^d(I, w'w w_2, \mu, \psi_{00}),
\end{align*}
so our functional equations may be expressed in the form
\begin{align*}
	& \Gamma_e(w'ww_2, \mu) \sigma_{\mu^{w'w}_1-\mu^{w'w}_2}(\abs{n_2}) \Sigmachi{d}{++} W^d(I, w w_2, \mu^{w'}, \psi_{00})^{-1} W^d(y, w, \mu^{w'}, \psi_n) \\
	& = \Gamma_E(w''w_2, \mu) \sigma_{\mu^{w''}_1-\mu^{w''}_2}(\abs{n_2}) \Sigmachi{d}{++} W^d(I, w'w w_2, \mu, \psi_{00})^{-1} W^d(y, w'', \mu, \psi_n),
\end{align*}
when $n_1=0\ne n_2$, and
\begin{align*}
	& \Gamma_e(w'ww_3, \mu) \sigma_{\mu^{w'w}_2-\mu^{w'w}_3}(\abs{n_1}) \Sigmachi{d}{++} W^d(I, w w_3, \mu^{w'}, \psi_{00})^{-1} W^d(y, w, \mu^{w'}, \psi_n) \\
	&= \Gamma_E(w''w_3, \mu) \sigma_{\mu^{w''}_2-\mu^{w''}_3}(\abs{n_1}) \Sigmachi{d}{++} W^d(I, w'w w_3, \mu, \psi_{00})^{-1} W^d(y, w'', \mu, \psi_n),
\end{align*}
when $n_1 \ne 0 = n_2$.
Since we have already shown the matrix $M(w,\mu)$ factors over the generators, we need only consider the functional equations coming from the Weyl elements $w'=w_2,w_3$.
Thus it suffices to show the functional equations for the following triples $(w', w, w'')$:
\[ (w_2, w_2, w_2), (w_2, w_4, w_l), (w_3, w_2, w_4), n_1=0\ne n_2, \]
\[ (w_3, w_3, w_3), (w_3, w_5, w_l), (w_2, w_3, w_5), n_1 \ne 0 = n_2. \]
The case $(w_3, w_4,w_2)$, for example, follows from $(w_3, w_2, w_4)$ and $M^d(w_3,\mu) M^d(w_3,\mu^{w_3}) = \Sigmachi{d}{++}$.

In the cases
\[ (w_2, w_4, w_l), (w_3, w_2, w_4), n_1=0\ne n_2, \]
\[ (w_3, w_5, w_l), (w_2, w_3, w_5), n_1 \ne 0 = n_2, \]
$w''$ is given by multiplication in the Weyl group, $w''=w'w$, so the $n$ and $y$ dependence transforms trivially, and we may reduce to the simpler equations
\[ \Sigmachi{d}{++} W^d(I, w w_2, \mu^{w'}, \psi_{00})^{-1} W^d(I, w, \mu^{w'}, \psi_{yn}) = \Sigmachi{d}{++} W^d(I, w''w_2, \mu, \psi_{00})^{-1} W^d(I, w'', \mu, \psi_{yn}), \]
when $n_1=0\ne n_2$, and
\[ \Sigmachi{d}{++} W^d(I, w w_3, \mu^{w'}, \psi_{00})^{-1} W^d(I, w, \mu^{w'}, \psi_{yn}) = \Sigmachi{d}{++} W^d(I, w''w_3, \mu, \psi_{00})^{-1} W^d(I, w'', \mu, \psi_{yn}), \]
when $n_1 \ne 0 = n_2$.
These cases may be verified directly, as in the $n=0$ case.

It remains to check
\[ (w_2, w_2, w_2), n_1=0\ne n_2, \]
\[ (w_3, w_3, w_3), n_1 \ne 0 = n_2. \]
These require an application of the functional equation of the classical Whittaker function \eqref{eq:ClassWhittFE} in combination with $\sigma_{-s}(n) = n^{-s}\sigma_s(n)$.
In the first case $(w_2,w_2,w_2), n_1=0$, it is sufficient to show
\[ \Sigmachi{d}{++} \paren{\frac{\Gamma_\R\paren{1-\mu_1+\mu_2}}{\Gamma_\R\paren{\mu_1-\mu_2}} \mathcal{W}^d(0, \mu_1-\mu_2) \Dtildek{d}{i}} = \Sigmachi{d}{++} \paren{\pi^{\mu_1-\mu_2} \Gamma^d_\mathcal{W}(\mu_2-\mu_1,\pm1)}, \]
and checking the even entries of the two diagonal matrices in parentheses verifies the result.
A similar argument handles the remaining case.

We note one last property of the $M(s,\mu)$ matrix:
\[ \trans{\wbar{M^d(w,\mu)}} M^d(w,\mu) = \Sigmachi{d}{++} \text{ for } \Re(\mu)=0. \]
To see this, it is sufficient to check $w=w_2,w_3$, and recognize that $\wbar{\mu} = -\mu$ for $\Re(\mu)=0$.

\subsubsection{The functional equations of the Eisenstein series}
\label{sect:MinEisenFEs}
For some $w\in W$, consider the difference of the two Eisenstein series
\[ f(g) := E^d(g,\mu)-M^d(w,\mu) E^d(g,\mu^w), \]
defined on the zero-free region \eqref{eq:FEset}.

Define the inner product on $\Gamma\backslash G$ in the usual manner,
\[ \innerprod{f_1,f_2} = \int_{\Gamma\backslash G} f_1(g) \wbar{\trans{f_2(g)}} dg, \]
then the $L^2$ norm of $f$ may be written
\[ \norm{f}_2^2 = \innerprod{f,f} = I_1(\mu) - I_2(\mu), \]
where
\begin{align*}
	I_1(\mu') = \innerprod{E^d(g,\mu'), f}, \qquad I_2(\mu') = \innerprod{M^d(w,\mu') E^d(g,(\mu')^w), f}.
\end{align*}

We may take the fundamental domain for $\Gamma\backslash G$ to lie inside a Siegel domain \cite[Prop. 1.3.2]{Gold01}, so that $y_1, y_2 \gg 1$ in the inner product, and since we have already shown the degenerate Fourier coefficients of $f$ are zero, proposition \ref{eq:FourierExpBound} implies that $f$ has super-polynomial decay as $y_i \to \infty$.
Then applying proposition \ref{eq:FourierExpBound} to each Eisenstein series separately shows that $I_1$ and $I_2$ converge to meromorphic functions on $\Re(\mu'_j-\mu'_k) \ge 0$ and $\Re({\mu'}_j^{w^{-1}}-{\mu'}_k^{w^{-1}}) \ge 0$, for each $j < k$, respectively.

By analytic continuation, we may investigate, say, $I_1(\mu')$ in the region of absolute convergence $\Re(\mu'_1-\mu'_2), \Re(\mu'_2-\mu'_3) > 1$.
On this region, the trivial bound \eqref{eq:TrivMinEisenBd} on the Eisenstein series implies that the combined sum-integral of
\[ I_1(\mu') = \int_{\Gamma\backslash G} \sum_{\gamma\in U(\Z)\backslash\Gamma} I^d(\gamma g, \mu') \wbar{\trans{f(g)}} dg \]
converges absolutely, so we may apply unfolding:
\[ I_1(\mu') = \int_{Y^+} p_{\rho+\mu'}(y) \int_{U(\Z)\backslash U(\R)} \wbar{\trans{f(xy)}} dx \, dy, \]
but the $x$ integral is zero since this is the $(0,0)$ Fourier coefficient.
Thus we conclude the $I_1(\mu')=0$ for all $\mu'$ by analytic continuation.

Similarly, $I_2(\mu')=0$, so we conclude that $\norm{f}_2^2=0$, and hence $M^d(w,\mu) E^d(g,\mu^w)$ gives the meromorphic continuation of $E^d(g,\mu)$ to $\Re({\mu'}_j^{w^{-1}}-{\mu'}_k^{w^{-1}}) \ge 0$, $j < k$.

We have proved
\begin{prop}
\label{prop:MinEisenFEs}
	The matrix-valued Eisenstein series $\zeta_E(\mu) E^d(g,\mu)$ has holomorphic continuation to all $\mu$, except for possible simple poles at $\mu_i-\mu_j=\pm 1$, with functional equations given by
	\[ E^d(g,\mu) = M^d(w,\mu) E^d(g,\mu^w), \qquad w \in W. \]
	The matrices $M^d(w,\mu)$ are orthogonal in the sense that
	\[ \trans{\wbar{M^d(w,\mu)}} M^d(w,\mu) = \Sigmachi{d}{++} \text{ for } \Re(\mu)=0. \]
\end{prop}
Multiplying by $\zeta_E(\mu)$ serves to cancel the poles coming from the matrix $M^d(w,\mu)$.

Computing the Fourier coefficients of both sides, this implies the functional equations of the even Whittaker functions,
\begin{cor}
\label{cor:EvenWhittFEs}
	The matrix-valued Whittaker function $\Sigmachi{d}{++} W^d(g, w_l, \mu, \psi_{1,1})$ has analytic continuation to all $\mu$ with functional equations given by
	\[ \Sigmachi{d}{++} W^d(g, w_l, \mu, \psi_{1,1}) = \paren{\Gamma_E(w,\mu) W^d(I,w,\mu,\psi_{00})} \Sigmachi{d}{++} W^d(g, w_l, \mu^w, \psi_{1,1}), \qquad w \in W. \]
\end{cor}

\subsection{The $1,1,1$ constant term of the spectral expansion}
\label{sect:111const}
Suppose $f:\Gamma\backslash G \to \C$ is Schwartz-class and orthogonal to the residues of the Eisenstein series, then we wish to expand the minimal parabolic constant term of $f$, meaning
\[ f^{111}(g) := \int_{U(\Z)\backslash U(\R)} f(ug) du, \]
into an integral of minimal parabolic Eisenstein series.

We may expand the $K$ part of $f^{111}$ using \eqref{eq:Kexpand}, and applying Mellin inversion gives
\begin{align*}
	f^{111}(xyk) =& \sum_{d=0}^\infty (2d+1) \Tr\Bigl(\WigDMat{d}(k) \frac{1}{(2\pi i)^2} \int_{\Re(\mu) = (-2,0,2)} p_{\rho+\mu}(y) \\
	& \qquad \int_{Y^+} \int_{U(\Z)\backslash U(\R)} \int_{K} f(x'y'k') \wbar{\trans{\WigDMat{d}(k')}} dk' \, dx' \, p_{\rho-\mu}(y') dy' \, d\mu \Bigr) \\
	=& \sum_{d=0}^\infty \frac{(2d+1)}{(2\pi i)^2} \int_{\Re(\mu) = (-2,0,2)} \Tr\Bigl(p_{\rho+\mu}(y) \WigDMat{d}(k) \int_{\Gamma\backslash G} f(g') \wbar{\trans{E^d(g', -\wbar{\mu})}} dg' \Bigr) \, d\mu.
\end{align*}

Now using the orthogonality of $f$ with the residues of $E^d$ to shift back to the lines \forcetextnewline $\Re(\mu)=0$, we have
\begin{align}
\label{eq:111constShifted}
	f^{111}(xyk) =& \frac{1}{\abs{V} \abs{W}} \sum_{d=0}^\infty \frac{(2d+1)}{(2\pi i)^2} \int_{\Re(\mu) = 0} \Tr\Bigl(\abs{V} \sum_{w\in W} p_{\rho+\mu^w}(y) \WigDMat{d}(k) \\
	& \qquad \int_{\Gamma\backslash G} f(g') \wbar{\trans{E^d(g', \mu^w)}} dg' \Bigr) \, d\mu. \nonumber
\end{align}
We apply the functional equations of $E^d$, again using the cycle-invariance of the trace,
\begin{align*}
	f^{111}(xyk) =& \frac{1}{\abs{V} \abs{W}} \sum_{d=0}^\infty \frac{(2d+1)}{(2\pi i)^2} \int_{\Re(\mu) = 0} \Tr\Bigl(\abs{V} \sum_{w\in W} p_{\rho+\mu^w}(y) \wbar{\trans{M^d(w^{-1},\mu^w)}} \WigDMat{d}(k) \\
	& \qquad \int_{\Gamma\backslash G} f(g') \wbar{\trans{E^d(g', \mu)}} dg' \Bigr) \, d\mu.
\end{align*}

The orthogonality and $\Sigmachi{d}{++}$-invariance of $M^d$ on $\Re(\mu)=0$ imply
\[ \wbar{\trans{M^d(w^{-1},\mu^w)}} = \wbar{\trans{M^d(w^{-1},\mu^w)}} \wbar{\trans{M^d(w,\mu)}} M^d(w,\mu) = M^d(w,\mu), \]
since $M^d(I,\mu)=\Sigmachi{d}{++}$.
So we may reconstruct the constant term of the Eisenstein series,
\begin{align*}
	f^{111}(xyk) =& \frac{1}{\abs{V} \abs{W}} \sum_{d=0}^\infty \frac{(2d+1)}{(2\pi i)^2} \int_{\Re(\mu) = 0} \Tr\Bigl(\abs{V} \sum_{w\in W} p_{\rho+\mu^w}(y) M^d(w,\mu) \WigDMat{d}(k) \\
	& \qquad \int_{\Gamma\backslash G} f(g') \wbar{\trans{E^d(g', \mu^w)}} dg' \Bigr) \, d\mu \\
	=& \frac{1}{\abs{V} \abs{W}} \sum_{d=0}^\infty \frac{(2d+1)}{(2\pi i)^2} \int_{\Re(\mu) = 0} \Tr\Bigl(E^{d,111}(xyk, \mu) \int_{\Gamma\backslash G} f(g') \wbar{\trans{E^d(g', \mu)}} dg' \Bigr) \, d\mu,
\end{align*}
that is, for $f_0$ as in Theorem \ref{thm:ContResSpectralExpand}, $f-f_0$ has zero minimal parabolic constant term.

It remains to show $f_0$ is square-integrable, but if we Fourier expand $E^d(g, \mu)$, then either the Whittaker function has rapid decay in $y_i$, or the integral over $\mu$ is has rapid decay in that $y_i$ by Mellin inversion.
(Using the fact that the $g'$ integral is essentially the Mellin transform of $f$ in the $y$ coordinates.)

\section{The maximal parabolic Eisenstein series}
\label{sect:MaxPara}
The maximal parabolic Eisenstein series are attached to Maass cusp forms in higher weight on $GL(2)$, so to define the Eisenstein series, we must first consider the higher-weight structure of $GL(2)$ cusp forms.
This is given in the paper of Duke, Friedlander and Iwaniec \cite{DFI01}.

\subsection{\texorpdfstring{Spectral basis of $L^2(SL(2,\Z)\backslash SL(2,\R))$}{Spectral basis of L2(SL(2,Z)\textbackslash SL(2,R))}}
\label{sect:GL2Basis}
We start with the Iwasawa decomposition of $g \in SL(2,\R)$,
\[ g = r \Matrix{y&x\\&1} \Matrix{\cos\theta&-\sin\theta\\\sin\theta&\cos\theta}, \qquad r,y\in\R^+, x\in\R, 0\le\theta<2\pi. \]
Now a Maass form for $SL(2,\Z)$ of weight $k\in \Z$ is a function of the form $\phi(x+iy) e^{-ik\theta}$ (using the isomorphism between $SL(2,\R)/SO(2,\R)$ and the complex upper half-plane) which is an eigenfunction of the Laplacian and left-invariant by $\gamma \in SL(2,\Z)$ in the sense that if
\[ \gamma \Matrix{y&x\\&1} = r \Matrix{y^*&x^*\\&1} \Matrix{\cos\theta^*&-\sin\theta^*\\\sin\theta^*&\cos\theta^*}, \]
then
\[ \phi(x^*+iy^*) e^{-ik(\theta^*+\theta)} = \phi(x+iy) e^{-ik\theta}. \]
In an abuse of terminology, we tend to refer to $\phi$ as the Maass form (this matches the terminology of \cite{DFI01}, but conflicts somewhat with our notation on $GL(3)$).
Note that $k$ here is necessarily even, since $-I \in SL(2,\Z)$.
If $\phi$ additionally satisfies $\int_0^1 \phi(x+iy) dx = 0$, we call it a Maass cusp form.
Such a function $\phi$ has a $GL(2)$ Fourier-Whittaker expansion
\begin{align}
\label{eq:SL2Fourier}
	\phi(x+iy) = \sum_\pm \sum_{n\ge 1} \rho_\phi(\pm n) \e{\pm nx} W_{\pm\frac{k}{2},\mu_\phi}(4\pi n y).
\end{align}

An orthonormal basis of such forms can be constructed from the spherical Maass cusp forms and holomorphic modular cusp forms by applying the weight raising and lowering operators as in \cite{DFI01}.
Each spherical Maass or holomorphic modular cusp form of minimal weight $\kappa \ge 0$ has a unique image in the basis of weight $k$ Maass forms for each $\abs{k} \ge \kappa$, and we would like to construct this image by giving its Fourier-Whittaker expansion using the Hecke eigenvalues of the form.
This amounts to specifying how the Fourier-Whittaker coefficients $\rho_\phi(\pm n)$ vary with the weight $k$.
We assume that all forms are Hecke eigenfunctions and that the spherical Maass forms are eigenfunctions of the parity/conjugation operator
\[ \phi(-x+iy) = \varepsilon \, \phi(x+iy), \qquad \varepsilon=\pm 1. \]
Thus we need only specify the values of $\rho_\phi(\pm 1)$, since the others may be obtained from
\[ \rho_\phi(\pm n) = \rho_\phi(\pm 1) \lambda_\Phi(n) n^{-\frac{1}{2}}, \qquad n \ge 1\]
where $\lambda_\Phi(n)$ are the Hecke eigenvalues of the minimal-weight ancestor of $\phi$.
(The need for the factor $n^{-\frac{1}{2}}$ can be seen in \cite[(4.82)]{DFI01} for holomorphic forms and \cite[9.235.2]{GradRyzh} for spherical Maass forms.)

One might worry that our primary reference \cite{DFI01} does not literally apply, as they assume a primitive character on a congruence subgroup, but it is simple to verify that the referenced results, especially their section 4, still hold.
The simplest explanation here is that information on the Whittaker functions and raising and lowering operators (i.e. the Archimedean place) is independent of the discrete subgroup and associated character (i.e. the non-Archimedean places).

If $\phi$ is the image in weight $k$ of a spherical Maass form $\Phi$, then from \cite[(4.77),(4.70)]{DFI01} (using $s=\frac{1}{2}+\mu_\Phi$ and ignoring the sign of $\alpha(s;k)$), \cite[9.235.2]{GradRyzh} and \cite[(4.4)]{Val01}, we know
\[ \rho_\phi(1)^{-2} = \frac{2 \, L(1, \AdSq \Phi)}{\cos (\pi \mu_\Phi)} \abs{\frac{\Gamma\paren{\frac{1}{2}+\mu_\Phi+\frac{\abs{k}}{2}}}{\Gamma\paren{\frac{1}{2}+\mu_\Phi-\frac{\abs{k}}{2}}}}, \qquad \rho_\phi(-1) = \varepsilon_\Phi \frac{\Gamma\paren{\frac{1}{2}+\mu_\Phi+\frac{k}{2}}}{\Gamma\paren{\frac{1}{2}+\mu_\Phi-\frac{k}{2}}} \rho_\phi(1), \]
where $\varepsilon_\Phi=\pm 1$ is the parity of the spherical form and $\mu_\phi=\mu_\Phi\in i\R$ is the spectral parameter.
(Note that $\cos (\pi \mu_\Phi) \ge 1$.)
It will be necessary to maintain a consistent choice for $\sgn (\Im (\mu_\Phi))$.

Similarly, if $\phi$ is the image of the holomorphic modular form $\Phi$ of (even) weight $\kappa > 0$, then we have $\mu_\phi=\mu_\Phi = \frac{\kappa-1}{2}$ and
\[ \rho_\phi(\sgn(k))^{-2} = \frac{\frac{\pi}{3} L(1,\AdSq \phi)}{(4\pi)^\kappa} \Gamma\paren{\frac{\abs{k}+\kappa}{2}} \Gamma\paren{\frac{2+\abs{k}-\kappa}{2}}, \qquad \rho_\phi(-\sgn(k)) =0, \]
for $\abs{k} \ge \kappa$, using \cite[(5.101)]{IK}, \cite[(4.83)]{DFI01}.
For later computations, we define $\varepsilon_\Phi=1$ for these forms.
A note on terminology:  We will frequently use the term ``holomorphic modular form'', but to be precise, we are referring to the image $y^{\kappa/2} \Phi(x+iy) e^{-i\kappa\theta}$ in the Maass forms of weight $\kappa > 0$.

We are allowed to multiply the form by a complex number of unit modulus, and so we take $\sgn(\rho_\phi(1))=(-1)^{k/2}$ for spherical Maass forms, and $\sgn(\rho_\phi(\sgn(k)))=(-1)^{k/2}$ for holomorphic modular forms.
The forms of negative weight are constructed directly from the positive weight forms by applying a conjugation operator, which we will discuss very precisely in the next section.

\subsection{Construction of the Eisenstein series}
\label{sect:MaxParaConstruct}
Set $P_{21} = \set{\SmallMatrix{*&*&*\\ *&*&*\\ 0&0&1}}$, $P_{12} = \set{\SmallMatrix{1&*&*\\ 0&*&*\\ 0&*&*}}$.
The vector-valued maximal parabolic Eisenstein series attached to an $SL(2,\Z)$ Hecke-Maass form $\phi$ of weight $m'$ are given by
\begin{align}
\label{eq:MaxParaDef}
	E^d_{m'}(g, \phi, \mu_1) =& \sum_{\gamma \in P_{21}(\Z) \backslash \Gamma} I^d_{m'}(\gamma g, \phi,\mu_1), \\
	I^d_{m'}(xyk, \phi, \mu_1) =& \phi(x_2+iy_2) (y_1^2 y_2)^{\frac{1}{2}+\mu_1} \WigDRow{d}{m'}(k), \nonumber
\end{align}
and
\begin{align*}
	\wtilde{E}^d_{m'}(g, \phi, \mu_1) =& \sum_{\gamma \in P_{12}(\Z) \backslash \Gamma} \wtilde{I}^d_{m'}(\gamma g, \phi,\mu_1), \\
	\wtilde{I}^d_{m'}(xyk, \phi,\mu_1) =& \phi(-x_1+iy_1) (y_1 y_2^2)^{\frac{1}{2}+\mu_1} \dualWigDRow{d}{m'}(k).
\end{align*}

The function $I^d_{m'}(xyk, \phi,\mu_1)$ is left-invariant under $\gamma \in SL(2,\Z)$ (embedded in the upper left coordinates in $SL(3,\Z)$), because if
\[ \gamma \Matrix{y_2&x_2\\&1} = r \Matrix{y_2^*&x_2^*\\&1} \Matrix{\cos\theta^*&-\sin\theta^*\\\sin\theta^*&\cos\theta^*}, \]
then
\begin{align*}
	I^d_{m'}(\gamma xyk, \phi,\mu_1) =& \phi(x_2^*+iy_2^*) (r^2 y_1^2 y_2^*)^{\frac{1}{2}+\mu_1} e^{-i m' \theta^*} \WigDRow{d}{m'}(k) \\
	=& \phi(x_2+iy_2) (y_1^2 y_2)^{\frac{1}{2}+\mu_1} \WigDRow{d}{m'}(k) \\
	=& I^d_{m'}(xyk,\phi,\mu_1),
\end{align*}
by the left-translation property of the weight $m'$ Maass form $\phi$, and
\[ r^2 y_1^2 y_2^* = \det y_1\gamma\Matrix{y_2&x_2\\&1} = y_1^2 y_2. \]

For a spherical Maass or holomorphic modular form $\Phi$, we define a diagonal matrix $\Phi^d$ whose entries $\Phi_{m',m'}^d$, $\abs{m'} \le d$ are the images of $\Phi$ in the orthonormal basis of weight $m'$ (hence the entries for odd $m'$ are necessarily zero).
We specify the relation between positive and negative weights with some care:
For the image in negative weights of even spherical Maass forms, use the conjugation operator $\phi(x+iy) \mapsto \phi(-x+iy)$, but for the odd spherical Maass forms, we use $\phi(x+iy) \mapsto -\phi(-x+iy)$.
Collectively, we say the forms have sign $\varepsilon_\Phi=\pm 1$ and obey the conjugation rule $\phi(x+iy) \mapsto \varepsilon_\Phi \phi(-x+iy)$.
For holomorphic modular forms, we have arbitrarily chosen $\varepsilon_\Phi = 1$, which is allowable since the image of a holomorphic modular form in the basis of weight zero forms is zero.
Thus in each case $\Phi^d_{-m',-m'}(-x+iy)=\varepsilon_\Phi \Phi^d_{m',m'}(x+iy)$, so the matrix satisfies
\begin{align}
\label{eq:PhidFE}
	\Phi^d(-x+iy) = \varepsilon_\Phi \WigDMat{d}(\vpmpm{+-}) \Phi^d(x+iy) \WigDMat{d}(\vpmpm{+-}),
\end{align}
by the explicit form of $\WigDMat{d}(\vpmpm{+-})$.

Then we may similarly collect the row vectors $E^d_{m'}(g, \Phi_{m',m'}^d, \mu_1)$ into a matrix, which we denote $E^d(g, \Phi, \mu_1)$.
Alternately, one may construct the matrix-valued Eisenstein series from the matrix-valued function $I^d(g, \Phi, \mu_1)$, defined in the appropriate manner.
Note that most of the entries in this matrix-valued Eisenstein series will be zero as only the even weights survive and rows of weight smaller in absolute value than the minimal weight for $\Phi$ will also be zero.

Define the involution $\iota:G\to G$ by
\begin{align}
\label{eq:iotaDef}
	g^\iota =& w_l \trans{(g^{-1})} w_l.
\end{align}
The functions $I^d$ and $\wtilde{I}^d$ are explicitly related by
\[ I^d(g, \Phi, \mu_1) = \wtilde{I}^d(g^\iota, \Phi, \mu_1), \]
and $(P_{12})^\iota=P_{21}$.
Thus the relationship between $E^d(g, \Phi, \mu_1)$ and the equivalent $\wtilde{E}^d(g, \Phi, \mu_1)$ is also given by the involution
\begin{align}
\label{eq:MaxEDualFE}
	E^d(g, \Phi, \mu_1) = \wtilde{E}^d(g^\iota, \Phi, \mu_1).
\end{align}

\subsection{Normalizations and the Fourier expansion of $I^d$}
\label{sect:MaxParaNorms}
It will be easiest to deal with the Hecke-normalized form, so for $\Phi$ a spherical Maass cusp form, we set
\begin{align*}
	\Gamma^d_{\Phi,m',m'} =& \pi^{-\frac{1}{2}+\mu_\Phi} i^{-m'} \abs{\frac{\cos(\pi \mu_\Phi) \Gamma\paren{\frac{1}{2}+\mu_\Phi-\frac{\abs{m'}}{2}}}{2 L(1,\AdSq \Phi) \Gamma\paren{\frac{1}{2}+\mu_\Phi+\frac{\abs{m'}}{2}}}}^{1/2} \Gamma\paren{\frac{1}{2}-\mu_\Phi+\frac{m'}{2}},
\end{align*}
and for $\Phi$ a holomorphic modular form of weight $\kappa > 0$ and sign $\varepsilon_\Phi = (-1)^{\alpha_\Phi}$,
\begin{align*}
	\Gamma^d_{\Phi,m',m'} =& \frac{(2\pi)^{\kappa} i^{-m'}}{\pi \sqrt{\frac{\pi}{3} L(1,\AdSq \Phi)}} \times \piecewise{\sqrt{\left. \Gamma\paren{\frac{2+\abs{m'}-\kappa}{2}} \middle / \Gamma\paren{\frac{\abs{m'}+\kappa}{2}} \right.} & \If \abs{m'}\ge\kappa, \\ 1 & \Otherwise.}
\end{align*}
As usual $\Gamma^d_\Phi$ is the diagonal matrix with the given entries.
Then the matrices for the $L^2$-normalized form $\Phi$ and what we term the ``Hecke-normalized'' form $\Phi_H$ are related by
\[ \Gamma^d_\Phi \Phi_H^d = \Phi^d. \]
We also define $\what{\Gamma}^d_\Phi$ by $\what{\Gamma}^d_\Phi = \Gamma^d_\Phi$ for spherical Maass forms $\Phi$, and 
\begin{align*}
	\what{\Gamma}^d_{\Phi,m',m'} =& \frac{(2\pi)^{\kappa} i^{-m'}}{\pi \sqrt{\frac{\pi}{3} L(1,\AdSq \Phi)}} \times \piecewise{\sqrt{\left. \Gamma\paren{\frac{\abs{m'}+\kappa}{2}} \middle / \Gamma\paren{\frac{2+\abs{m'}-\kappa}{2}} \right.} & \If \abs{m'}\ge\kappa, \\ 1 & \Otherwise,}
\end{align*}
for holomorphic modular forms.
By the functional equation of the classical Whittaker function, we could equally well have defined the Hecke normalization by
\[ \what{\Gamma}^d_\Phi \Phi_{\what{H}}^d = \Phi^d, \]
with the understanding $\mu_\Phi = \frac{1-\kappa}{2}$ for holomorphic modular forms (in place of $\mu_\Phi = \frac{\kappa-1}{2}$).
It will become important that $\wbar{\Gamma^d_\Phi} \what{\Gamma}^d_\Phi$ is a scalar matrix.

Set
\begin{align}
\label{eq:MaxEisenMu}
	\mu'=\mu'(\mu_\Phi,\mu_1)=(\mu_1-\mu_\Phi,\mu_1+\mu_\Phi,-2\mu_1),
\end{align}
then the Fourier expansion of $I^d(xy, \Phi_H, \mu_1)$ may be computed from \eqref{eq:SL2Fourier}
\begin{align}
\label{eq:IdFour}
	I^d(xy, \Phi_H, \mu_1) =& \sum_\pm \sum_{n\in\N} (\pm 1)^{\alpha_\Phi} \lambda_\Phi(n) n^{\mu_\Phi} \Xi^d \WigDMat{d}(\vpmpm{+-}) W^d(xy,w_2,\mu', \psi_{(0,\pm n)}),
\end{align}
where the sign of $\Phi$ is $\varepsilon_\Phi = (-1)^{\alpha_\Phi}$, and
\begin{align}
\label{eq:XiDef}
	\Xi^d =& \tfrac{1}{2}\paren{\WigDMat{d}(\vpmpm{++})+\WigDMat{d}(\vpmpm{-+})}
\end{align}
is the diagonal matrix with entries $\Xi^d_{m',m'} = \delta_{2|m'}$.
The choice of signs of $\rho_\phi(\pm1)$ as described at the end of section \ref{sect:GL2Basis} (which appears in the definitions of $\Gamma^d_\Phi$, and $\what{\Gamma}^d_\Phi$) compensates for a factor $\Dtildek{d}{i}$ which would otherwise appear in \eqref{eq:IdFour}.
This simplifies the computations of the Fourier-Whittaker coefficients slightly.

For $\Phi$ a holomorphic modular form of weight $\kappa > 0$, it is highly important to note that the $m'$-th row ($m'$ even) of
\[ \WigDMat{d}(\vpmpm{+-}) W^d(I,w_2,\mu', \psi_{(0,n_2 y_2)}) = \Dtildek{d}{i}\mathcal{W}^d(-n_2 y_2, 1-\kappa)  \]
is zero unless $\sgn(n_2) m' \ge \kappa$.
This is due to the pole of the gamma function in the denominator of \eqref{eq:Weval}, and it aligns perfectly with the fact that these rows of (the Fourier-Whittaker coefficients of) $\Phi^d$ are zero.
Thus there is no need to include a matrix in \eqref{eq:IdFour} to enforce this.
On the other hand, the normalization $\Phi_{\what{H}}^d$ does not have this property, but this will be handled trivially below.

\subsection{\texorpdfstring{A Bruhat-type decomposition for $P_{21}(\Z)\backslash \Gamma$}{A Bruhat-type decomposition for P21(Z)\textbackslash{Gamma}}}
\label{sect:MaxParaBruhat}
In analogy with the Bruhat decomposition, we may write an element of $\Gamma$ as one of
\[ \Matrix{a_1&b_1&b_2\\a_2&b_3&b_4\\c&d_1&d_2} = \Matrix{1&0&\frac{a_1}{c}\\&1&\frac{a_2}{c}\\&&1}\Matrix{\frac{1}{c} u\\&c} w_4 \Matrix{1&\frac{d_1}{c}&\frac{d_2}{c}\\&1&0\\&&1}, \quad u=\Matrix{c b_1-d_1 a_1&c b_2-d_2 a_1\\c b_3-d_1 a_2&c b_4-d_2 a_2}, \]
\[ \Matrix{a_1&b_1&b_2\\a_2&b_3&b_4\\0&c&d_1} = \Matrix{1&0&\frac{b_1}{c}\\&1&\frac{b_3}{c}\\&&1}\Matrix{\frac{1}{c} u\\&-c} w_3 \Matrix{1&0&0\\&1&\frac{d_1}{c}\\&&1}, \quad u=\Matrix{-c a_1&d_1 b_1-c b_2\\-c a_2&d_1 b_3-c b_4}, \]
or
\[ \Matrix{a_1&b_1&b_2\\a_2&b_3&b_4\\0&0&c} = \Matrix{1&0&\frac{b_2}{c}\\&1&\frac{b_4}{c}\\&&1} \Matrix{\frac{1}{c} u\\&c} I I, \quad u=c \Matrix{a_1&b_1\\a_2&b_3}, \]
where $c\ne 0$.
Note that, in each case, $\det u=c$.
If we translate $u$ on the left by an element of $SL(2,\Z)$, we may assume $u$ is upper triangular; in particular, the set of $P_{21}(\Z) \gamma$ with $\gamma$ running through the $I$, $w_3$, and $w_4$ Weyl cells of $\Gamma$ covers the space $P_{21}(\Z)\backslash \Gamma$.

Now suppose that $P_{21}(\Z) \gamma_1 = P_{21}(\Z) \gamma_2$ with each $\gamma_i$ in one of the $I$, $w_3$, and $w_4$ Weyl cells of $\Gamma$.
Immediately we see that $\gamma_1$ and $\gamma_2$ have the same bottom row, and hence also belong to the same Weyl cell.
The Bruhat decomposition, say $\gamma_1=b_1cvw b_2$ is unique if we require $b_2\in \wbar{U}_w(\Q)$, and from the above, we can see $b_2$ depends only on the bottom row.
Thus, if also $\gamma_2 = b_1'c'v'w' b_2'$, then $P_{21}(\Z) \gamma_1 = P_{21}(\Z) \gamma_2$ is equivalent to $w=w'$, $b_2=b_2'$ and $(b_1cv) (b_1'c'v')^{-1} \in P_{21}(\Z)$, but the set of upper triangular matrices in $P_{21}(\Z)$ is exactly $\wbar{V}_2 \, U(\Z)$, with
\[ \wbar{V}_2 = \set{I,\vpmpm{-+}}. \]
In summary, a complete set of representatives for $P_{21}(\Z)\backslash \Gamma$ is given by the representatives of $\wbar{V}_2 \, U(\Z)\backslash \Gamma_w$ as $w$ runs through $I$, $w_3$, and $w_4$ where $\Gamma_w \subset \Gamma$ is the Weyl cell for $w$.

We wish to be more explicit:
For the Weyl cell of the identity, a set of representatives is given by
\[ V_2 = \set{I, \vpmpm{+-}}. \]
For the $w_3$ Weyl cell, a set of representatives is given by
\[ \Matrix{1&0&0\\&1&-\frac{\wbar{d_1}}{c_1}\\&&1}\Matrix{1\\&\frac{1}{c_1}\\&&c_1}v\,w_3\Matrix{1&0&0\\&1&\frac{d_1}{c_1}\\&&1}, \]
with $v\in V_2$, $c_1 \in \N$, $d_1\in\Z$, $(d_1,c_1)=1$, $\wbar{d_1} d_1 \equiv 1 \pmod{c_1}$.
For the $w_4$ Weyl cell, we may express the Bruhat decomposition via the Pl\"ucker coordinates by
\[ \Matrix{1&\frac{-\wbar{C_2} C_1}{B_2}&\frac{\wbar{C_2}}{A_1}\\&1&-\frac{\wbar{C_1}}{A_1/B_2}\\&&1}\Matrix{\frac{1}{B_2}\\&\frac{B_2}{A_1}\\&&A_1}v\,w_4\Matrix{1&-\frac{C_2}{B_2}&\frac{C_1}{A_1}\\&1&0\\&&1}, \]
where $v\in V_2$, $C_1,C_2 \in \Z$, and $A_1 \in \N$, $0 < B_2 | A_1$ with $(C_1,A_1/B_2)=(C_2,B_2)=1$.
$\wbar{C_2}$ and $\wbar{C_1}$ are integers which are the multiplicative inverses of $C_2$ and $C_1$ modulo $B_2$ and $A_1/B_2$, respectively.
Note the minor notational discrepancy in the $x_3$ coordinate of $b_1$ here has to do with the multiplication in $U(\Z)$, but it is well-defined.

This analysis is equivalent to Miyazaki's Lemma 5.1 \cite{Miya01}, but we prefer the above formulation.

\subsection{Fourier coefficients}
\label{sect:MaxParaFourier}
The sum in the Eisenstein series can be viewed as the finite part of an intertwining operator, so our computations here specialize and make explicit those of Shahidi \cite{Shah01} at the unramified places.
The Fourier integral can also be viewed as the place at infinity of an intertwining operator, but the computations here are not to be found in Shahidi's book \cite{Shah01}, as our Eisenstein series are ramified at infinity.
The paper \cite{Miya01} gives an essentially identical computation, in a radically different notation.
We will compute all of the Fourier coefficients, and use this computation to determine the constant terms in the next section.

As before, it is sufficient to consider the Fourier coefficients at $g=y\in Y^+$, so we compute
\begin{align*}
	\rho_{E^d}(n,y,\Phi_H,\mu_1) :=& \int_{U(\Z)\backslash U(\R)} E^d(uy, \Phi_H, \mu_1) \wbar{\psi_n(u)} du.
\end{align*}
Splitting the sum over $P_{21}\backslash\Gamma$ in the definition \eqref{eq:MaxParaDef} over the Weyl cells of the previous section, we compute the Fourier coefficients of each cell separately, say
\begin{align*}
	\rho_{E^d,w}(n,y,\Phi_H,\mu_1) :=& \int_{U(\Z)\backslash U(\R)} \sum_{\gamma\in\wbar{V}_2U(\Z)\backslash\Gamma_w} I^d(\gamma uy, \Phi_H, \mu_1) \wbar{\psi_n(u)} du,
\end{align*}
for $w=I,w_3,w_4$.

\subsubsection{The identity Weyl cell}
From \eqref{eq:PhidFE}, we have that
\begin{align*}
	I^d(\vpmpm{+-} xyk, \Phi_H, \mu_1) = I^d((\vpmpm{+-} x \vpmpm{+-}) y, \Phi_H, \mu_1) \WigDMat{d}(\vpmpm{+-}k)= \varepsilon_\Phi \WigDMat{d}(\vpmpm{+-}) I^d(xyk, \Phi_H, \mu_1).
\end{align*}
By the $SL(2,\Z)$ invariance (recall $\Phi^d_{m',m'}=0$ for odd $m'$), we have
\begin{align*}
	I^d(\vpmpm{+-} xyk, \Phi_H, \mu_1) = I^d(xyk, \Phi_H, \mu_1) = \WigDMat{d}(\vpmpm{-+}) I^d(xyk, \Phi_H, \mu_1).
\end{align*}
Thus we may convert the sum on $V_2$ to the projection operator on $V$,
\begin{align}
\label{eq:MaxEisenVSum}
	\sum_{v\in V_2} I^d(v g, \Phi_H, \mu_1) = 2\Sigma^d_{+\varepsilon_\Phi} I^d(g, \Phi_H, \mu_1).
\end{align}

Applying \eqref{eq:IdFour}, we immediately have the Fourier coefficients of the Weyl cell of the identity,
\begin{align}
\label{eq:MaxEisenIdFour}
	\rho_{E^d,I}(n,y,\Phi_H,\mu_1) = 2 \varepsilon_\Phi \sgn(n_2)^{\alpha_\Phi} \lambda_\Phi(\abs{n_2}) \abs{n_2}^{\mu_\Phi} \Sigma^d_{+\varepsilon_\Phi} W^d(y,w_2,\mu', \psi_n)
\end{align}
if $n_2 \ne 0 = n_1$, and 0 otherwise.

\subsubsection{The $w_3$ Weyl cell}
From the explicit Bruhat-type decomposition, we may write these terms as
\begin{align*}
	\rho_{E^d, w_3}(n,y,\Phi_H,\mu_1) =& 2\Sigma^d_{+\varepsilon_\Phi} \int_{U(\Z)\backslash U(\R)} \sum_{c_1\in\N} \sum_{\substack{d_1\in\Z\\(d_1,c_1)=1}} I^d\paren{x^* y^* k^*, \Phi_H, \mu_1} \wbar{\psi_n(u)} du,
\end{align*}
where
\[ x^* y^* k^* \equiv \Matrix{1\\&\frac{1}{c_1}\\&&c_1}\,w_3\Matrix{1&0&0\\&1&\frac{d_1}{c_1}\\&&1} uy \pmod{\R^+}. \]
From the Iwasawa decomposition, we see that the only occurance of $u_3$ is in $x_2^*$, so the $u_3$ integral is
\[ \int_0^1 \Phi_H^d\paren{c_1 u_3-c_1 \paren{c_1 u_1+d_1} u_2+iy_2^*} du_3 = 0. \]
(This is not to say these terms sum to the zero function, but rather their contribution to the Fourier expansion is picked up by the sum over $U(\Z)\backslash SL(2,\Z)$.)

\subsubsection{The $w_4$ Weyl cell}
From the explicit Bruhat-type decomposition, we may write these terms as
\begin{align*}
	\rho_{E^d, w_4}(n,y,\Phi_H,\mu_1) =& \sum_{\substack{A_1 \in \N\\0 < B_2 | A_1}} \sum_{\substack{C_1,C_2 \in \Z\\(C_1,A_1/B_2)=(C_2,B_2)=1}} \sum_{v\in V_2} \int_{U(\Z)\backslash U(\R)} I^d(x^* y^* k^*, \Phi_H, \mu_1) \wbar{\psi_n(u)} du,
\end{align*}
where
\[ x^* y^* k^* \equiv \Matrix{1&\frac{-\wbar{C_2} C_1}{B_2}&0\\&1&0\\&&1}\Matrix{\frac{1}{B_2}\\&\frac{B_2}{A_1}\\&&A_1}v\,w_4\Matrix{1&-\frac{C_2}{B_2}&\frac{C_1}{A_1}\\&1&0\\&&1} uy \pmod{\R^+}, \]
with $\wbar{C_2}C_2 \equiv 1 \pmod{B_2}$.

Sending
\[ \Matrix{1&-\frac{C_2}{B_2}&\frac{C_1}{A_1}\\&1&0\\&&1} u \mapsto u, \]
and similarly conjugating $\SmallMatrix{1&\frac{-\wbar{C_2} C_1}{B_2}&0\\&1&0\\&&1}$ across, changing the sign of $C_1$ as necessary,
gives
\begin{align}
\label{eq:Beforen2zero}
	\rho_{E^d, w_4}(n,y,\Phi_H,\mu_1) =& 2 \Sigma^d_{+\varepsilon_\Phi} \sum_{a_1,a_2 \in \N} \sum_{\substack{C_1\summod{a_1}\\C_2\summod{a_2}\\(C_1,a_1)=(C_2,a_2)=1}} \e{-n_2 \frac{C_2}{a_2}-n_1\frac{\wbar{C_2} C_1}{a_1}} \\
& \qquad \int_{[0,a_2]\times\R^2} I^d\paren{\Matrix{\frac{1}{a_2}\\[2pt]&\frac{1}{a_1}\\&&a_1 a_2}w_4 uy, \Phi_H, \mu_1} \wbar{\psi_n(u)} du, \nonumber
\end{align}
after setting $a_1 = \frac{A_1}{B_2}$, $a_2=B_2$.

If $n_1=0$, then the $x_2$ coordinate of the argument of $I^d$ is
\[ \frac{a_1}{a_2}\paren{u_1-\frac{u_2 u_3}{y_2^2+u_2^2}}, \]
hence the integral is zero by cuspidality of $I^d$, and we now assume $n_1\ne 0$.
Similarly, by sending $u_1 \mapsto a_2 u_1$, the integral is zero unless $a_1|(a_2 n_1)$, so the modular sums are well-defined, and we may remove the $\wbar{C_2}$ using $(C_2, a_1/(a_1,n_1))=1$.

Inserting the Fourier expansion of $I^d$ from \eqref{eq:IdFour} and evaluating the $u_1$ integral, the finite and infinite places separate,
\[ \rho_{E^d, w_4}(n,y,\Phi_H,\mu_1) = 2 \varepsilon_\Phi \sgn(n_1)^{\alpha_\Phi} \rho_\text{fin} \Sigma^d_{+\varepsilon_\Phi} \rho_\text{inf}. \]
Writing out the infinite place, i.e. the $u_2$ and $u_3$ integrals, gives an integral of the form \eqref{eq:WhittPartialEval}, so this becomes
\begin{align}
\label{eq:MaxEisenFourInfPart}
	\rho_\text{inf} =& y_1^{1-\mu_1+\mu_\Phi} y_2^{1-2\mu_1} \WigDMat{d}(\vpmpm{-+}) W^d(I,\mu',\psi_{yn}) = \WigDMat{d}(\vpmpm{-+}) W^d(y,\mu',\psi_n).
\end{align}

The finite part is the sum
\begin{align*}
	\rho_\text{fin} =& \abs{n_1}^{\mu_\Phi} \sum_{\substack{a_1,a_2,a_3 \in \N\\a_1 a_3=a_2 \abs{n_1}}} a_1^{-1-3\mu_1} a_2^{-1-3\mu_1} \sum_{\substack{C_1\summod{a_1}\\C_2\summod{a_2}\\(C_1,a_1)=(C_2,a_2)=1}} \e{n_2 \frac{C_2}{a_2}+n_1\frac{C_1}{a_1}} \lambda_\Phi(a_3).
\end{align*}
where $\sgn(n_1) a_3$ was the value of the index $\pm n$ occuring in \eqref{eq:IdFour}.
These sums are evaluated in \cite[Lemma 5.5]{Miya01}:
For $n_1 \ne 0=n_2$, we have
\begin{align*}
	\rho_\text{fin} =& \abs{n_1}^{\mu_\Phi} \lambda_\Phi(\abs{n_1}) \frac{L(\Phi,3\mu_1)}{L(\Phi,1+3\mu_1)},
\end{align*}
where
\[ L(\Phi,s) = \sum_{n=1}^\infty \frac{\lambda_\Phi(n)}{n^s} \]
is the usual Hecke $L$-function attached to $\Phi$.

If we take the Satake parameters $\lambda_\Phi(p)=a_\Phi(p)+b_\Phi(p)$ with $a_\Phi(p) b_\Phi(p) = 1$, and define by multiplicativity
\begin{align}
\label{eq:MaxEisenHecke}
	& \lambda_E((p^\alpha, p^\beta), \Phi, \mu_1) = p^{3\mu_1(\alpha-\beta)} S_{\alpha,\beta}\paren{a_\Phi(p) p^{-3\mu_1}, b_\Phi(p) p^{-3\mu_1}},
\end{align}
and $\lambda_E(n, \Phi, \mu_1) = \prod_p \lambda_E((p^{v_p(n_1)}, p^{v_p(n_2)}), \Phi, \mu_1)$ (note $\lambda_E((1, 1), \Phi, \mu_1)=1$), we have
\begin{align*}
	\rho_\text{fin} =& \abs{n_1}^{\mu_\Phi}\abs{n_2}^{-3\mu_1} \frac{\lambda_E(n, \Phi, \mu_1)}{L(\Phi,1+3\mu_1)},
\end{align*}
when $n_1 n_2 \ne 0$.

It is possible to show the more explicit form
\begin{align}
\label{eq:MaxEisenHecke2}
	& \lambda_E((p^\alpha, p^\beta), \Phi, \mu_1) =\\
	& \frac{p^{-3\mu_1(\alpha+1)}\lambda_\Phi(p^{\beta})+p^{3\mu_1(\beta+1)}\lambda_\Phi(p^{\alpha}) -\lambda_\Phi(p^{\alpha+1})\lambda_\Phi(p^{\beta})-\lambda_\Phi(p^{\alpha})\lambda_\Phi(p^{\beta+1})+\lambda_\Phi(p)\lambda_\Phi(p^{\alpha})\lambda_\Phi(p^{\beta})}{p^{3\mu_1}+p^{-3\mu_1}-\lambda_\Phi(p)} \nonumber
\end{align}
Notice the coefficients $\lambda_E(n, \Phi, \mu_1)$ are actually the $GL(3)$ Hecke eigenvalues, c.f. \cite[10.9.3]{Gold01}:
\begin{align}
\label{eq:MaxEisenHeckeRels1}
	\lambda_E((p^{\alpha}, 1), \Phi, \mu_1) =& \sum_{a_1=0}^{\alpha} \lambda_\Phi(p^{a_1}) p^{3\mu_1(\alpha-a_1)}, \\
\label{eq:MaxEisenHeckeRels2}
	\lambda_E((1,p^{\beta}), \Phi, \mu_1) =& \sum_{a_1=0}^{\alpha} \lambda_\Phi(p^{a_1}) p^{-3\mu_1(\alpha-a_1)}, \\
\label{eq:MaxEisenHeckeRels3}
	\lambda_E((p^{\alpha}, p^{\beta}), \Phi, \mu_1) =& \lambda_E((p^{\alpha}, 1), \Phi, \mu_1)\lambda_E((1,p^{\beta}), \Phi, \mu_1) \\
	& \qquad-\lambda_E((p^{\alpha-1}, 1), \Phi, \mu_1)\lambda_E((1,p^{\beta-1}), \Phi, \mu_1), \nonumber
\end{align}
again defining $\lambda_E((p^{-1}, 1), \Phi, \mu_1)=\lambda_E((1,p^{-1}), \Phi, \mu_1)=0$.

The Fourier-Whittaker coefficients in the normalization of \eqref{eq:FWCoefsDef} are then
\begin{align}
\label{eq:MaxEisenFour1}
	\rho_{E^d}(w_2, n,\Phi_H,\mu_1) =& 2 \delta_{n_2\ne 0} \varepsilon_\Phi \sgn(n_2)^{\alpha_\Phi} \lambda_\Phi(\abs{n_2}) \abs{n_2}^{-\mu_1} \Sigma^d_{+\varepsilon_\Phi}, \\
\label{eq:MaxEisenFour2}
	\rho_{E^d}(w_l, n,\Phi_H,\mu_1) =& 2 \delta_{n_1\ne 0} \sgn(n_1)^{\alpha_\Phi} \frac{\wtilde{n_1}^{\mu_1}\wtilde{n_2}^{-\mu_1}}{L(\Phi,1+3\mu_1)} \Sigma^d_{+\varepsilon_\Phi} \\
	& \qquad \times \piecewise{\lambda_\Phi(\abs{n_1}) L(\Phi,3\mu_1) & \If n_1 \ne 0 = n_2,\\ \lambda_E(n, \Phi, \mu_1) & \If n_1 n_2 \ne 0,} \nonumber
\end{align}
and all other cases are zero.
Recall $\delta_{w_2,n}=0$ unless $n_1=0$.

\subsection{Constant terms}
\label{sect:MaxParaConstTerms}
It will be useful for later sections to reconstitute the constant terms from the Fourier expansion.
The $1,1,1$ constant term is zero.
The $2,1$ constant term is
\begin{align}
	E^{d,21}(xy, \Phi_H, \mu_1) := \int_{\R^2} E^d\paren{\SmallMatrix{1&x_2&u_3\\&1&u_1\\&&1} y, \Phi_H, \mu_1} du_1\, du_3 = 2 \Sigma^d_{+\varepsilon_\Phi} I^d\paren{\SmallMatrix{1&x_2\\&1\\&&1} y, \Phi_H, \mu_1},
\end{align}
from \eqref{eq:MaxEisenFour1} and \eqref{eq:IdFour} (keeping in mind the slight difference in notation of \eqref{eq:FWCoefsDef}).

For the $1,2$ constant term, we point out that the degenerate character Whittaker function may be written as
\begin{align*}
	& W^d(xy, w_l, \mu, \psi_{(n_1,0)}) = (y_1 y_2^2)^{\mu_3} \WigDMat{d}(\vpmpm{+-} w_l) \mathcal{W}^d(0,\mu_2-\mu_3) \\
	& \qquad \times \WigDMat{d}(w_3) \mathcal{W}^d(0,\mu_1-\mu_3) \WigDMat{d}(w_l) W^d(\vpmpm{+-}x^\iota y^\iota \vpmpm{+-}, w_2, \mu, \psi_{(0,n_1)}) \WigDMat{d}(\vpmpm{+-}) \WigDMat{d}(w_l),
\end{align*}
and the $1,2$ constant term is
\begin{align*}
	E^{d,12}(xy, \Phi_H, \mu_1) :=& \int_{\R^2} E^d\paren{\SmallMatrix{1&u_2&u_3\\&1&x_1\\&&1} y, \Phi_H, \mu_1} du_2\, du_3 \\
	=& 2 \frac{L(\Phi,3\mu_1)}{L(\Phi,1+3\mu_1)} \Sigma^d_{+\varepsilon_\Phi} \sum_{n_1\ne 0} \sgn(n_1)^{\alpha_\Phi} \lambda_\Phi(\abs{n_1}) \Xi^d \abs{n_1}^{\mu_\Phi} W^d(xy,w_l,\mu',\psi_{(n_1,0)}).
\end{align*}
Setting
\begin{align}
\label{eq:tildeMdef}
	\wtilde{M}^d(\Phi, \mu_1) =& \frac{L(\Phi,3\mu_1)}{L(\Phi,1+3\mu_1)} \Sigma^d_{+\varepsilon_\Phi} \WigDMat{d}(w_l) \mathcal{W}^d(0,3\mu_1+\mu_\Phi) \WigDMat{d}(w_3) \mathcal{W}^d(0,3\mu_1-\mu_\Phi) \WigDMat{d}(w_l),
\end{align}
we have
\begin{align}
	E^{d,12}(xy, \Phi, \mu_1) =& 2 \wtilde{M}^d(\Phi, \mu_1) I^d\paren{\SmallMatrix{1\\&1&x_1\\&&1}^\iota y^\iota w_l, \Phi, -\mu_1},
\end{align}
after working through the $V$ dependence.

\subsection{Initial continuation}
Assuming we have a zero-free region for $L(\Phi, 1+s)$ of the form \eqref{eq:FormalZeroFree} (for the present purposes, no uniformity in $\Phi$ is necessary and \cite[Theorem 5.10]{IK} is sufficient), and taking for granted the fact that $L(\Phi, 3\mu_1)$ has no poles on $\Re(\mu_1) > 0$ \cite[Section 8]{DFI01}, the analytic continuation follows from the absolute convergence of the Fourier expansion given by Proposition \ref{eq:FourierExpBound}.
Notice that the polynomial boundedness of the Fourier coefficients follows from \eqref{eq:MaxEisenHeckeRels1}-\eqref{eq:MaxEisenHeckeRels3} and any polynomial bound (e.g. \cite[(14.54), (16.56)]{DFI01}) on the $\lambda_\Phi(n)$.

\subsection{Functional equations}
As we did for the minimal Eisenstein series, we show the functional equations by examining the constant terms.
We wish to show
\begin{align*}
	E^{d,12}(xy, \Phi_H, \mu_1) =& \wtilde{M}^d(\Phi, \mu_1) E^{d,21}((xy)^\iota w_l, \Phi_H, -\mu_1), \\
	E^{d,21}(xy, \Phi_H, \mu_1) =& \wtilde{M}^d(\Phi, \mu_1) E^{d,12}((xy)^\iota w_l, \Phi_H, -\mu_1),
\end{align*}
and it suffices to show
\begin{align}
\label{eq:MaxFEToShow1}
	\wtilde{M}^d(\Phi, \mu_1) \Sigma^d_{+\varepsilon_\Phi} =& \wtilde{M}^d(\Phi, \mu_1), \\
\label{eq:MaxFEToShow2}
	\wtilde{M}^d(\Phi, \mu_1) \wtilde{M}^d(\Phi, -\mu_1) &= \Sigma^d_{+\varepsilon_\Phi}.
\end{align}
The first statement follows by conjugating each $v\in V$ to the left and using $\Sigma^d_{+\varepsilon_\Phi}\Sigma^d_{+\varepsilon_\Phi}=\Sigma^d_{+\varepsilon_\Phi}$.

The functional equation of the $L$-function may be expressed in the form
\[ L_\infty(\Phi,s) L(\Phi,s) = \eta_\Phi L_\infty(\Phi,1-s) L(\Phi,1-s), \]
where
\[ L_\infty(\Phi,s) = \Gamma_\R\paren{s+\mu_\Phi} \Gamma_\R\paren{1+s+\mu_\Phi}, \qquad \eta_\Phi = i^\kappa, \]
for $\Phi$ a holomorphic modular form of weight $\kappa > 0$,
\[ L_\infty(\Phi,s) = \Gamma_\R\paren{s+\mu_\Phi} \Gamma_\R\paren{s-\mu_\Phi}, \qquad \eta_\Phi = 1, \]
for $\Phi$ an even spherical Maass form (i.e. $\varepsilon_\Phi=1$), and
\[ L_\infty(\Phi,s) = \Gamma_\R\paren{1+s+\mu_\Phi} \Gamma_\R\paren{1+s-\mu_\Phi}, \qquad \eta_\Phi = -1, \]
for $\Phi$ an odd spherical Maass form (i.e. $\varepsilon_\Phi=-1$).
These may be found in \cite[(8.11), (8.17)]{DFI01}

For \eqref{eq:MaxFEToShow2}, in case $\Phi$ is an even spherical Maass form or a holomorphic modular form (so that $\varepsilon_\Phi=1$), we are once again only concerned with the entries on the even rows (i.e. $m'$ even).
As in \eqref{eq:mathcalWFE}, we simply need to show
\begin{align}
\label{eq:MaxFEToShow3}
	& \frac{L_\infty(\Phi,1+3\mu_1) L_\infty(\Phi,1-3\mu_1)}{L_\infty(\Phi,3\mu_1) L_\infty(\Phi,-3\mu_1)} \\
& \qquad \times \frac{\Gamma_\R\paren{3\mu_1+\mu_\Phi}\Gamma_\R\paren{-3\mu_1-\mu_\Phi}\Gamma_\R\paren{3\mu_1-\mu_\Phi}\Gamma_\R\paren{-3\mu_1+\mu_\Phi}}{\Gamma_\R\paren{1+3\mu_1+\mu_\Phi}\Gamma_\R\paren{1-3\mu_1-\mu_\Phi}\Gamma_\R\paren{1+3\mu_1-\mu_\Phi}\Gamma_\R\paren{1-3\mu_1+\mu_\Phi}} \nonumber
\end{align}
is one.
It is.

In case $\Phi$ is an odd spherical Maass form, so that $\varepsilon_\Phi=-1$, we conjugate $\Sigma^d_{+\varepsilon_\Phi}$ past $\WigDMat{d}(w_l)$ in \eqref{eq:tildeMdef}, and by \eqref{eq:SigmadRels} and \eqref{eq:Sigmadminusplus}, we see that we are only concerned with the odd rows (i.e. $m'$ odd).
From \eqref{eq:Beval}, we need to show \eqref{eq:MaxFEToShow3} multiplied by
\[ \frac{\sin\paren{\pi \frac{3\mu_1+\mu_\Phi}{2}}}{\sin\paren{\pi\frac{1+3\mu_1+\mu_\Phi}{2}}} \frac{\sin\paren{\pi \frac{-3\mu_1+\mu_\Phi}{2}}}{\sin\paren{\pi\frac{1-3\mu_1+\mu_\Phi}{2}}} \frac{\sin\paren{\pi \frac{3\mu_1-\mu_\Phi}{2}}}{\sin\paren{\pi\frac{1+3\mu_1-\mu_\Phi}{2}}} \frac{\sin\paren{\pi \frac{-3\mu_1-\mu_\Phi}{2}}}{\sin\paren{\pi\frac{1-3\mu_1-\mu_\Phi}{2}}} \]
is one (the four factors $(-1)^{(m'+1)/2}$ cancel).
It is.

The functional equation follows by the same method as in the minimal parabolic case in section \ref{sect:MinEisenFEs} except that we unfold now to the space
\[ P_{21}(\Z)\backslash G/K \cong \paren{U(\Z)\backslash\set{\SmallMatrix{1&0&*\\&1&*\\&&1}}} \times (SL(2,\Z)\backslash SL(2,\R)/SO(2,\R)) \times \R^+, \]
and the integral of $f(g) :=E^d(xy, \Phi, \mu_1) - \wtilde{M}^d(\Phi, \mu_1) E^d((xy)^\iota w_l, \Phi, -\mu_1)$ over the space $U(\Z)\backslash\set{\SmallMatrix{1&0&*\\&1&*\\&&1}}$ is zero.

Set $\mu'' = -\wbar{\mu'}$.
On $\Re(\mu_1)=0$ and for a spherical Maass form this leaves $\mu''=\mu'$, but for a holomorphic modular form, this has the effect of replacing $\mu_\Phi \mapsto -\mu_\Phi$.
We define $\what{M}^d(\Phi, \mu_1)$ as in \eqref{eq:tildeMdef}, but with $\mu'$ replaced with $\mu''$.
Then we arrive at
\begin{prop}
\label{prop:MaxFE}
	The matrix-valued Eisenstein series $L(\Phi,1+3\mu_1) E^d(xy, \Phi, \mu_1)$ has holomorphic continuation to all $\mu_1\in\C$, with functional equations given by
	\begin{align*}
		E^d(xy, \Phi_H, \mu_1) =& \wtilde{M}^d(\Phi, \mu_1) E^d((xy)^\iota w_l, \Phi_H, -\mu_1), \\
		E^d(xy, \Phi_{\what{H}}, \mu_1) =& \what{M}^d(\Phi, \mu_1) E^d((xy)^\iota w_l, \Phi_{\what{H}}, -\mu_1).
	\end{align*}

	The matrices $\wtilde{M}^d(\Phi, \mu_1)$ are orthogonal in the sense that
	\[ \trans{\wbar{\what{M}^d(\Phi, \mu_1)}} \wtilde{M}^d(\Phi, \mu_1) = \Sigma^d_{+\varepsilon_\Phi} \text{ for } \Re(\mu_1)=0. \]
\end{prop}
The orthogonality follows from the previous computations, upon noticing that
\[ \trans{\wbar{\what{M}^d(\Phi, \mu_1)}} = \wtilde{M}^d(\Phi, -\mu_1), \]
when $\Re(\mu_1)=0$.

\subsection{The $2,1$ and $1,2$ constant terms of the spectral expansion}
We assume that $f:SL(3,\Z)\backslash SL(3,\R) \to \C$ is Schwartz class, and consider the $2,1$ constant term in the form
\[ f^{21}(xyk) := \int_0^1 \int_0^1 f\paren{\SmallMatrix{1&x_2&u_3\\&1&u_1\\&&1}yk} du_1 \, du_2, \]
and $f^{21}_c := (f-f_0)^{21}$ with $f_0$ as in Theorem \ref{thm:ContResSpectralExpand}.
As before, we expand the $K$-part; then applying Mellin inversion gives
\begin{align*}
	f^{21}_c(xyk) =& \sum_{d=0}^\infty (2d+1) \sum_{m'=-d}^d \frac{1}{2\pi i} \int_{\Re(\mu_1) = -2} (y_1^2 y_2)^{\frac{1}{2}+\mu_1} \WigDRow{d}{m'}(k) \int_0^\infty \int_0^1 \int_0^1 \\
	& \qquad \int_{K} (f-f_0)\paren{\SmallMatrix{1&x_2&u_3\\&1&u_1\\&&1} \SmallMatrix{\sqrt{t_1 y_2}\\&\sqrt{t_1/y_2}\\&&1}k'} \wbar{\trans{\WigDRow{d}{m'}(k')}} dk' \, du_3 \, du_1 \, t_1^{\frac{1}{2}-\mu_1} \frac{dt_1}{t_1^2} \, d\mu_1.
\end{align*}
The function of $x_2$ and $y_2$ at each $m'$ transforms as a weight $m'$ Maass form which is cuspidal by definition of $f_0$, so we may expand it over an orthonormal basis of Hecke eigenforms $\wtilde{\mathcal{S}}_{m'}$.
\begin{align*}
	& f^{21}_c(xyk) = \\
	& \sum_{d=0}^\infty (2d+1) \sum_{m'=-d}^d \sum_{\phi\in\wtilde{\mathcal{S}}_{m'}} \frac{1}{2\pi i} \int_{\Re(\mu_1) = -2} (y_1^2 y_2)^{\frac{1}{2}+\mu_1} \phi(x_2+iy_2) \WigDRow{d}{m'}(k) \int_0^\infty \int_{\mathcal{F}_2} \int_0^1 \int_0^1 \\
	& \int_{K} (f-f_0)\paren{\SmallMatrix{1&x_2'&x_3'\\&1&x_1'\\&&1} \SmallMatrix{\sqrt{t_1 t_2}\\&\sqrt{t_1/t_2}\\&&1}k'} \wbar{\phi(x_2'+it_2)} \wbar{\trans{\WigDRow{d}{m'}(k')}} dk' \, dx_3' \, dx_1' \, \frac{dx_2'\,dt_2}{t_2^2} t_1^{\frac{1}{2}-\mu_1} \frac{dt_1}{t_1^2} \, d\mu_1,
\end{align*}
where $\mathcal{F}_2$ is a fundamental domain for $SL(2,\Z)\backslash SL(2,\R)/SO(2,\R)$.

Now substituting $(t_1,t_2) \mapsto (y_1^2 y_2, y_2)$ gives
\begin{align*}
	f^{21}_c(xyk) =& 2 \sum_{d=0}^\infty \sum_{m'=-d}^d \sum_{\phi\in\wtilde{\mathcal{S}}_{m'}} \frac{(2d+1)}{2\pi i} \int_{\Re(\mu_1) = -2} (y_1^2 y_2)^{\frac{1}{2}+\mu_1} \phi(x_2+iy_2) \WigDRow{d}{m'}(k) \int_{P_{21}\backslash G} \\
	& \qquad f\paren{g'} \wbar{\phi(x_2'+iy_2')} \wbar{\trans{\WigDRow{d}{m'}(k')}} ({y_1'}^2 y_2')^{\frac{1}{2}-\mu_1} dg' \, d\mu_1,
\end{align*}
with the usual coordinates $g'=x'y'k'$.
By folding, this is
\begin{align*}
	f^{21}_c(xyk) =& 2 \sum_{d=0}^\infty \sum_{m'=-d}^d \sum_{\phi\in\wtilde{\mathcal{S}}_{m'}} \frac{(2d+1)}{2\pi i} \int_{\Re(\mu_1) = 0} (y_1^2 y_2)^{\frac{1}{2}+\mu_1} \phi(x_2+iy_2) \WigDRow{d}{m'}(k) \\
	& \qquad \int_{\Gamma\backslash G} f\paren{g'} \wbar{\trans{E^d_{m',\cdot}(g', \phi, -\wbar{\mu_1})}} \, dg' \, d\mu_1,
\end{align*}
after shifting $\Re(\mu_1) \mapsto 0$, as we may.
We have dropped $f_0$ from this last equation since $E^d(g,\mu)$ and hence $f_0$ are orthogonal to functions which have no $1,1,1$ constant term, such as $E^d_{m',\cdot}(g', \phi, -\wbar{\mu_1})$.

Let $\mathcal{S}^d_{2,\text{cusp}}$ be the set of all minimal (non-negative) weight ancestors of each $\wtilde{\mathcal{S}}_{m'}$, $\abs{m'}\le d$.
Then we may collect the previous expression into the form
\begin{align*}
	f^{21}_c(xyk) =& 2 \sum_{d=0}^\infty \sum_{\Phi\in\mathcal{S}^d_{2,\text{cusp}}} \frac{(2d+1)}{2\pi i} \int_{\Re(\mu_1) = 0} \Tr\biggl((y_1^2 y_2)^{\frac{1}{2}+\mu_1} \Phi^d(x_2+iy_2) \WigDMat{d}(k) \\
	& \qquad \int_{\Gamma\backslash G} f\paren{g'} \wbar{\trans{E^d(g', \Phi, \mu_1)}} \, dg' \biggr) d\mu_1.
\end{align*}

We switch to Hecke normalizations
\begin{align*}
	f^{21}_c(xyk) =& 2 \sum_{d=0}^\infty \sum_{\Phi\in\mathcal{S}^d_{2,\text{cusp}}} \frac{(2d+1)}{2\pi i} \int_{\Re(\mu_1) = 0} \Tr\biggl(\wbar{\Gamma^d_\Phi} \what{\Gamma}^d_\Phi (y_1^2 y_2)^{\frac{1}{2}+\mu_1} \Phi_{\what{H}}^d(x_2+iy_2) \WigDMat{d}(k) \\
	& \qquad \int_{\Gamma\backslash G} f\paren{g'} \wbar{\trans{E^d(g', \Phi_H, \mu_1)}} \, dg' \biggr) d\mu_1,
\end{align*}
keeping in mind that $\wbar{\Gamma^d_\Phi} \what{\Gamma}^d_\Phi$ is a scalar matrix.

The equation \eqref{eq:PhidFE} implies
\begin{align}
\label{eq:MaxEisenVFE}
	E^d(g', \Phi, \mu_1)=\Sigma^d_{+\varepsilon_\Phi} E^d(g', \Phi, \mu_1)
\end{align}
in the same manner as it did for $I^d$.
Applying the cycle-invariance of the trace, we recognize the $2,1$ constant term of the Eisenstein series,
\begin{align}
\label{eq:f21Expand}
	f^{21}_c(g) =& \sum_{d=0}^\infty \sum_{\Phi\in\mathcal{S}^d_{2,\text{cusp}}} \frac{(2d+1)}{2\pi i} \int_{\Re(\mu_1) = 0} \Tr\biggl(\wbar{\Gamma^d_\Phi} \what{\Gamma}^d_\Phi E^{d,21}(g, \Phi_{\what{H}}, \mu_1) \\
	& \qquad \int_{\Gamma\backslash G} f\paren{g'} \wbar{\trans{E^d(g', \Phi_H, \mu_1)}} \, dg' \biggr) d\mu_1. \nonumber
\end{align}
(There is a sneaky point that when $\Phi_{\what{H}}$ is coming from a holomorphic modular form of weight $\kappa > 0$, we are not trying to attach meaning to the nonsensical rows $E^{d,21}_{m'}(g, \Phi_{\what{H}}, \mu_1)$ for $\abs{m'} < \kappa$ as these are zero in $E^d(g', \Phi_H, \mu_1)$ and so they drop out of the former by cycle invariance.)

At last, we return to $L^2$ normalizations to conclude that $f-f_0-f_1$, with $f_1$ as in Theorem \ref{thm:ContResSpectralExpand}, has no $2,1$ constant term.

We are left to show that the previous display is also lacking a $1,2$ constant term, but if we apply the construction up to \eqref{eq:f21Expand} to the function $\wtilde{f}(g) = f(g^\iota)$, we see that
\begin{align*}
	f^{12}_c(g) = \wtilde{f}^{21}_c(g^\iota)=& \sum_{d=0}^\infty \sum_{\Phi\in\mathcal{S}^d_{2,\text{cusp}}} \frac{(2d+1)}{2\pi i} \int_{\Re(\mu_1) = 0} \Tr\biggl(\wbar{\Gamma^d_\Phi} \what{\Gamma}^d_\Phi E^{d,21}(g^\iota, \Phi_{\what{H}}, \mu_1) \\
	& \qquad \int_{\Gamma\backslash G} f\paren{(g')^\iota} \wbar{\trans{E^d(g', \Phi_H, \mu_1)}} \, dg' \biggr) d\mu_1.
\end{align*}

Now applying \eqref{eq:MaxEisenVFE} in conjunction with the orthogonality from Proposition \ref{prop:MaxFE} gives
\begin{align*}
	E^d(g', \Phi, \mu_1)=\trans{\wbar{\what{M}^d(\Phi, -\mu_1)}} \wtilde{M}^d(\Phi, -\mu_1) E^d(g', \Phi, \mu_1),
\end{align*}
and then the functional equation gives
\begin{align*}
	E^d(g', \Phi, \mu_1)=\trans{\wbar{\what{M}^d(\Phi, -\mu_1)}} E^d((g')^\iota w_l, \Phi, \mu_1) \WigDMat{d}(w_l).
\end{align*}
We finish by the cycle-invariance and the fact that $\trans{\wbar{\WigDMat{d}(w_l)}} = \WigDMat{d}(w_l^{-1})$.
By the same argument as for the minimal parabolic part, $f_1$ will be square-integrable.

\section{The residual spectrum}
\label{sect:ResSpec}
Let us consider the contour shifting in section \ref{sect:111const} more carefully.
Dropping the assumption that $f$ is orthogonal to the residues of the Eisenstein series, we assume instead that $f$ is orthogonal to the constant function.
To obtain \eqref{eq:111constShifted}, we shifted contours from $\Re(\mu)=(-2,0,2)$ to $\Re(\mu)=0$.
Now we use three steps, shifting first to $\Re(\mu_1-\mu_2) = \eta$ for some small $\eta > 0$, while maintaining $\Re(\mu_2-\mu_3)=-2$.
This may be done through a change of variables, say $u=(\mu_1-\mu_2,\mu_2-\mu_3)$.
We possibly encounter simple poles in the Fourier coefficients of the Eisenstein series at $-\wbar{(\mu_1-\mu_2)}=1$.
For the shifted contour, now at $\Re(\mu)=(-\frac{2}{3}+\eta,-\frac{2}{3},\frac{4}{3})$, we again shift contours to $\Re(\mu_2-\mu_3)=0$, while maintaining $\Re(\mu_1-\mu_2)=\eta$.
Again we may have encountered poles at $-\wbar{(\mu_2-\mu_3)}=1$ and $-\wbar{(\mu_1-\mu_3)}=1$, which we may take to be distinct.
Finally, we shift to $\Re(\mu_1-\mu_2)=0$; the resulting integral of this third contour shift is \eqref{eq:111constShifted}.
The residues now have the form
\begin{align}
\label{eq:Res}
	& \sum_{i=1}^3 \sum_{d=0}^\infty \frac{(2d+1)}{2\pi i} \int_{\Re(s) = c_i} \Tr\Bigl(f_i(y) \WigDMat{d}(k) \int_{\Gamma\backslash G} f(g') \wbar{\trans{F^d_i(g', -\wbar{s})}} dg' \Bigr) \, ds,
\end{align}
where
\begin{align*}
	c_1 = -\frac{5}{6}, \qquad c_2=\frac{1}{6}+\eta, \qquad c_3=\frac{1}{6}-\eta,
\end{align*}
\begin{align*}
	f_1(y) =& y_1^{1+2s} y_2^{\frac{1}{2}+s}=p_{\rho+\mu_F^{w_2}}(y), & F^d_1(g,s) :=& \left. \res_{u_1=1} E^d(g, \mu) \right|_{u_2=-\frac{1}{2}+3s}, \\
	f_2(y)=& y_1^{\frac{1}{2}-s} y_2^{1-2s}=p_{\rho+\mu_F^{w_l}}(y), & F^d_2(g,s) :=& \left. \res_{u_2=1} E^d(g, \mu) \right|_{u_1=\frac{1}{2}+3s} = M^d(w_5, \mu_F^{w_4}) F^d_1(g,s), \\
	f_3(y)=& y_1^{\frac{1}{2}-s} y_2^{\frac{1}{2}+s}=p_{\rho+\mu_F^{w_5}}(y), & F^d_3(g,s) :=& \left. \res_{u_1=1-u_2} E^d(g, \mu) \right|_{u_2=\frac{1}{2}-3s} = M^d(w_3, \mu_F^{w_3}) F^d_1(g,s),
\end{align*}
with
\[ \mu_F = (\tfrac{1}{2}+s,-\tfrac{1}{2}+s,-2s). \]
We have used the functional equations of Proposition \ref{prop:MinEisenFEs}, with, for example, the fact that $\Gamma_E(w_3,\mu^{w_3})$, and hence $M(w_3,\mu^{w_3})$, are holomorphic at $\mu_1-\mu_2=1$, provided $\mu_2-\mu_3 \notin \Z$.

From section \ref{sect:MinEisenFour}, we have the Fourier coefficients of the function $F^d_1(g,\mu_1)$.
The poles occur in the $\Gamma_E(w,\mu)$, and we see
\begin{align*}
	\rho_{F^d_1}(w,n,s) = \abs{V} \zeta_0(w, s, \psi_n) p_{-\mu_F^w}(\wtilde{n}) \Sigmachi{d}{++},
\end{align*}
where for $n_1 n_2 \ne 0$, $\zeta_0(w, s, \psi_n)=0$; for $n_1=n_2=0$,
\begin{align*}
	\zeta_0(w, \mu_1, \psi_n) = \frac{6}{\pi^2} \times \piecewise{
1 & \If w=w_2, \\[5pt]
\frac{\zeta(\frac{1}{2}+3s)}{\zeta(\frac{3}{2}+3s)} & \If w=w_5, \\[5pt]
\frac{\zeta(-\frac{1}{2}+3s)}{\zeta(\frac{3}{2}+3s)} & \If w=w_l \\
0 &\Otherwise,}
\end{align*}
for $n_1=0, n_2\ne 0$,
\begin{align*}
	\zeta_0(w, \mu_1, \psi_n) =& \frac{6}{\pi^2} \delta_{w=w_l} \frac{\Gamma_\R\paren{\frac{1}{2}-3s}\zeta(\frac{1}{2}-3s)\sigma_{\frac{1}{2}-3s}(\abs{n_2})}{\Gamma_\R\paren{\frac{1}{2}+3s} \zeta(\tfrac{1}{2}+3s) \zeta(\tfrac{3}{2}+3s)} \\
	=& \frac{6}{\pi^2} \delta_{w=w_l} \frac{\sigma_{\frac{1}{2}-3s}(\abs{n_2})}{\zeta(\tfrac{3}{2}+3s)}
\end{align*}
and for $n_1\ne 0, n_2=0$,
\begin{align*}
	\zeta_0(w, \mu_1, \psi_n) =& \frac{6}{\pi^2} \delta_{w=w_5} \frac{\sigma_{-\frac{1}{2}-3s}(\abs{n_1})}{\zeta(\tfrac{3}{2}+3s)}.
\end{align*}

We shift the contours in \eqref{eq:Res} back to $\Re(s)=0$.
An easy approach to handle the poles of the $F^d_i(g,s)$ is to notice that the functions $\WigDMat{d}(k)$, $y_1^{1/3} y_2^{2/3} \WigDMat{d}(k)$, and $y_1^{2/3} y_2^{1/3} \WigDMat{d}(k)$ are not $\Gamma$ invariant except the first in the particular case $d=0$.
We would like to address the poles directly.
This will require some information about $\mathcal{W}^d(0,u)$ at three points.

First, from the explicit form \eqref{eq:Beval} of $\mathcal{W}_m(0,1)$, we see that $\Sigmachi{d}{++} \mathcal{W}^d(0,1) = \pi \Delta^d$, where $\Delta^d$ is the matrix with the single, central entry $\Delta^d_{0,0}=1$.
Second, the even entries of $\mathcal{W}^d(0,u)$ have simple poles at $u=0$ with residue $2 \cos\frac{\pi m}{2}$.  We may express this as
\[ \res_{u=0} \mathcal{W}^d(0,u) = \Dtildek{d}{i}+\Dtildek{d}{-i}=\WigDMat{d}(\vpmpm{+-} w_2)+\WigDMat{d}(\vpmpm{--} w_2). \]
Lastly, we need the central entry of
\[ \WigDMat{d}(w_3) \mathcal{W}^d(0,2) \WigDMat{d}(w_5) \vpmpm{+-}, \]
but if we expand the integral \eqref{eq:Wint} and apply the substitution $x=\cot \beta$, we may see this is
\[ 2\innerprod{\WigD{d}{0}{0}, \WigD{0}{0}{0}}_{L^2(K)} = 2\delta_{d=0}, \]
since $\frac{1}{\sqrt{2d+1}} \WigD{d}{0}{0}$ and $\WigD{0}{0}{0}=1$ are elements of an orthonormal basis.

$F^d_1(xy,s)$ has a possible pole at $s=\frac{1}{2}$ with residue
\begin{align*}
	\frac{8}{\pi^2 \zeta(3)} p_{\rho+\mu_F^{w_l}}(y) \Sigmachi{d}{++} W^d(I,w_l,\mu_F,\psi_{00}) &=\frac{16}{\zeta(3)} \delta_{d=0},
\end{align*}
and a possible pole at $s=\frac{1}{6}$ with residue
\begin{align*}
	\frac{48}{\pi^4} p_{\rho+\mu_F^{w_5}}(y) \Sigmachi{d}{++} W^d(I,w_5,\mu_F,\psi_{00}) +\frac{72}{\pi^4} p_{\rho+\mu_F^{w_l}}(y) \Sigmachi{d}{++} \res_{s=\frac{1}{6}} W^d(I,w_l,\mu_F,\psi_{00}) &= 0.
\end{align*}
(Here we used $\Sigmachi{d}{++} \WigDMat{d}(w_2) \Delta^d = \Sigmachi{d}{++} \Delta^d$.)
From similar computations, we see $F^d_2(xy,s)$ has a possible pole at $s=-\frac{1}{6}$ with residue
\begin{align*}
	\frac{48}{\pi^4} p_{\rho+\mu_F^{w_4 w_4}}(y) \Sigmachi{d}{++} W^d(I,w_4,\mu_F^{w_4},\psi_{00}) +\frac{72}{\pi^4} p_{\rho+\mu_F^{w_4 w_l}}(y) \Sigmachi{d}{++} \res_{s=-\frac{1}{6}} W^d(I,w_l,\mu_F^{w_4},\psi_{00}) &=0.
\end{align*}
The contour for $F^d_3(xyk,s)$ is already to the right of all of the poles.

Since $\mathcal{W}^d(0,1) = \Delta^d \mathcal{W}^d(0,1)$, the Whittaker functions
\[ \Sigmachi{d}{++} W^d(y,w_2,\mu_F,\psi_{00}) = \pi p_{\rho+\mu_F^{w_2}}(y) \Delta^d, \]
$\Sigmachi{d}{++} W^d(y,w_5,\mu_F,\psi_{(n_1,0)})$, and $\Sigmachi{d}{++} W^d(y,w_l,\mu_F,\psi_{(0,n_2)})$ are all left-invariant by $\Delta^d$, and so is $F^d_1(g,\mu_1)$.
Using this fact, and the cycle-invariance of the trace, the shifted integral from \eqref{eq:Res} is
\begin{align*}
	& \frac{1}{2} \sum_{d=0}^\infty \frac{(2d+1)}{2\pi i} \int_{\Re(s) = 0} \Tr\Bigl(H^d(y,s) \WigDMat{d}(k) \int_{\Gamma\backslash G} f(g') \wbar{\trans{F^d_1(g', s)}} dg' \Bigr) \, ds,
\end{align*}
where
\[ H^d(y,s) = p_{\rho+\mu_F^{w_2}}(y) \Delta^d +p_{\rho+\mu_F^{w_l}}(y) \Delta^d \wbar{\trans{M^d(w_5, \mu_F^{w_4})}}+p_{\rho+\mu_F^{w_5}}(y) \Delta^d \wbar{\trans{M^d(w_3,\mu_F^{w_3})}}. \]

Expanding the definitions of $M^d(w_5, \mu_F^{w_4})$ shows
\[ H^d(y,s) = \frac{\pi}{24} F^{d,111}_1(y,s), \]
and the contribution of the residual spectrum to \eqref{eq:111constShifted} will be
\begin{align}
\label{eq:ResContribF1}
	& \frac{\pi}{24} \sum_{d=0}^\infty \frac{(2d+1)}{2\pi i} \int_{\Re(s) = 0} \Tr\Bigl(F^d_1(g,s) \int_{\Gamma\backslash G} f(g') \wbar{\trans{F^d_1(g', s)}} dg' \Bigr) \, ds.
\end{align}

\subsection{The maximal Eisenstein series attached to the constant function}
We start with $\phi = \sqrt{\frac{3}{\pi}}$, the constant function on $SL(2,\Z)\backslash SL(2,\R)$.
In the style of sections \ref{sect:MaxParaConstruct} and \ref{sect:MaxParaNorms}, we define the Hecke-normalized $\Phi_H^d$ to be the diagonal matrix with the single nonzero entry $\Phi^d_{H,0,0}=1$, and $E^d(g,\Phi_H,\mu_1)$ in an identical manner.
The spectral parameter in the sense of \eqref{eq:MaxEisenMu} is $\mu_\Phi=-\frac{1}{2}$.

The Fourier coefficients can be obtained by computing as in section \ref{sect:MaxParaFourier}, and we see
\[ F^d_1(g,\mu_1) = \frac{12}{\pi} E^d(g,\Phi_H,\mu_1). \]
(If one believes such an equality should hold, it is easiest to check the constant by comparing the Weyl cell of the identity in $E^d(g,\Phi_H,\mu_1)$ to the $w_2$ constant term in $F^d_1(g,\mu_1)$.)

Comparing to \eqref{eq:ResContribF1}, the contribution of the residual spectrum in \eqref{eq:111constShifted} will be
\begin{align*}
	& \frac{6}{\pi} \sum_{d=0}^\infty \frac{(2d+1)}{2\pi i} \int_{\Re(\mu_1) = 0} \Tr\Bigl(E^d(g,\Phi_H,\mu_1) \int_{\Gamma\backslash G} f(g') \wbar{\trans{E^d(g',\Phi_H,\mu_1)}} dg' \Bigr) \, d\mu_1,
\end{align*}
and if $\Phi = \sqrt{\frac{6}{\pi}}\Phi_H$, then this is
\begin{align*}
	& \sum_{d=0}^\infty \frac{(2d+1)}{2\pi i} \int_{\Re(\mu_1) = 0} \Tr\Bigl(E^d(g,\Phi,\mu_1) \int_{\Gamma\backslash G} f(g') \wbar{\trans{E^d(g',\Phi,\mu_1)}} dg' \Bigr) \, d\mu_1.
\end{align*}

\section{Acknowledgements}
The author would like to thank Valentin Blomer, Stephen D. Miller, Xiaoqing Li, James Cogdell and Joseph Hundley for their comments and helpful discussions along the way.

\bibliographystyle{amsplain}

\bibliography{HigherWeight}

\providecommand{\bysame}{\leavevmode\hbox to3em{\hrulefill}\thinspace}
\providecommand{\MR}{\relax\ifhmode\unskip\space\fi MR }
\providecommand{\MRhref}[2]{%
  \href{http://www.ams.org/mathscinet-getitem?mr=#1}{#2}
}
\providecommand{\href}[2]{#2}
\begin{thebibliography}{10}

\bibitem{BL01}
L.C. Biedenharn, J.D. Louck, and P.A. Carruthers, \emph{Angular momentum in
  quantum physics: Theory and application}, Encyclopedia of Mathematics and its
  Applications, Cambridge University Press, 2009.

\bibitem{Val01}
V.~Blomer, \emph{Applications of the {K}uznetsov formula on {$GL(3)$}}, Invent.
  Math. \textbf{194} (2013), no.~3, 673--729.

\bibitem{Bump01}
D.~Bump, \emph{Automorphic forms on {${\rm GL}(3,{\bf R})$}}, Lecture Notes in
  Mathematics, vol. 1083, Springer-Verlag, Berlin, 1984.

\bibitem{Me01}
J.~Buttcane, \emph{On sums of {$SL(3,\Bbb{Z})$} {K}loosterman sums}, Ramanujan
  J. \textbf{32} (2013), no.~3, 371--419.

\bibitem{DFI01}
W.~Duke, J.~B. Friedlander, and H.~Iwaniec, \emph{The subconvexity problem for
  {A}rtin {$L$}-functions}, Invent. Math. \textbf{149} (2002), no.~3, 489--577.

\bibitem{Gold01}
D.~Goldfeld, \emph{Automorphic forms and {$L$}-functions for the group {${\rm
  GL}(n,\bold R)$}}, Cambridge Studies in Advanced Mathematics, vol.~99,
  Cambridge University Press, Cambridge, 2006, With an appendix by Kevin A.
  Broughan.

\bibitem{GradRyzh}
I.~S. Gradshteyn and I.~M. Ryzhik, \emph{Table of integrals, series, and
  products}, eighth ed., Elsevier/Academic Press, Amsterdam, 2015, Translated
  from the Russian, Translation edited and with a preface by Daniel Zwillinger
  and Victor Moll, Revised from the seventh edition [MR2360010].

\bibitem{ImaiTerras}
K.~Imai and A.~Terras, \emph{The {F}ourier expansion of {E}isenstein series for
  {${\rm GL}(3,\,{\bf Z})$}}, Trans. Amer. Math. Soc. \textbf{273} (1982),
  no.~2, 679--694.

\bibitem{IK}
H.~Iwaniec and E.~Kowalski, \emph{Analytic number theory}, American
  Mathematical Society Colloquium Publications, vol.~53, American Mathematical
  Society, Providence, RI, 2004.

\bibitem{Jac01}
H.~Jacquet, \emph{Fonctions de {W}hittaker associ\'ees aux groupes de
  {C}hevalley}, Bull. Soc. Math. France \textbf{95} (1967), 243--309.

\bibitem{Kn01}
A.W. Knapp, \emph{Lie groups beyond an introduction}, Progress in Mathematics,
  Birkh{\"a}user Boston, 2002.

\bibitem{Langlands01}
R.~P. Langlands, \emph{Eisenstein series}, Algebraic {G}roups and
  {D}iscontinuous {S}ubgroups ({P}roc. {S}ympos. {P}ure {M}ath., {B}oulder,
  {C}olo., 1965), Amer. Math. Soc., Providence, R.I., 1966, pp.~235--252.

\bibitem{Langlands02}
\bysame, \emph{On the functional equations satisfied by {E}isenstein series},
  Lecture Notes in Mathematics, Vol. 544, Springer-Verlag, Berlin-New York,
  1976.

\bibitem{ManIshOda}
H.~Manabe, T.~Ishii, and T.~Oda, \emph{Principal series {W}hittaker functions
  on {${\rm SL}(3,R)$}}, Japan. J. Math. (N.S.) \textbf{30} (2004), no.~1,
  183--226.

\bibitem{Miya01}
T.~Miyazaki, \emph{The {E}isenstein series for {$GL(3,{\bf Z})$} induced from
  cusp forms}, Abh. Math. Semin. Univ. Hambg. \textbf{82} (2012), no.~1, 1--41.

\bibitem{PS01}
I.~I. Pjateckij-{\v{S}}apiro, \emph{Euler subgroups}, Lie groups and their
  representations ({P}roc. {S}ummer {S}chool, {B}olyai {J}\'anos {M}ath.
  {S}oc., {B}udapest, 1971), Halsted, New York, 1975, pp.~597--620.

\bibitem{Sel01}
A.~Selberg, \emph{Harmonic analysis and discontinuous groups in weakly
  symmetric {R}iemannian spaces with applications to {D}irichlet series}, J.
  Indian Math. Soc. (N.S.) \textbf{20} (1956), 47--87.

\bibitem{Sel02}
\bysame, \emph{On discontinuous groups in higher-dimensional symmetric spaces},
  Contributions to function theory (internat. {C}olloq. {F}unction {T}heory,
  {B}ombay, 1960), Tata Institute of Fundamental Research, Bombay, 1960,
  pp.~147--164.

\bibitem{Sel03}
\bysame, \emph{Discontinuous groups and harmonic analysis}, Proc. {I}nternat.
  {C}ongr. {M}athematicians ({S}tockholm, 1962), Inst. Mittag-Leffler,
  Djursholm, 1963, pp.~177--189.

\bibitem{Shah01}
F.~Shahidi, \emph{Eisenstein series and automorphic {$L$}-functions}, American
  Mathematical Society Colloquium Publications, vol.~58, American Mathematical
  Society, Providence, RI, 2010.

\bibitem{Shal01}
J.~A. Shalika, \emph{The multiplicity one theorem for {${\rm GL}_{n}$}}, Ann.
  of Math. (2) \textbf{100} (1974), 171--193.

\bibitem{Stade01}
E.~Stade, \emph{Mellin transforms of {${\rm GL}(n,\Bbb R)$} {W}hittaker
  functions}, Amer. J. Math. \textbf{123} (2001), no.~1, 121--161.

\bibitem{VinTaht}
A.~I. Vinogradov and L.~A. Tahtad{\v{z}}jan, \emph{Theory of the {E}isenstein
  series for the group {${\rm SL}(3,\,{\bf R})$}\ and its application to a
  binary problem. {I}. {F}ourier expansion of the highest {E}isenstein series},
  Zap. Nauchn. Sem. Leningrad. Otdel. Mat. Inst. Steklov. (LOMI) \textbf{76}
  (1978), Analytic number theory and the theory of functions.

\end{thebibliography}

\end{document}